\documentclass[11pt,a4paper]{article}

\usepackage[T1]{fontenc}
\usepackage[utf8]{inputenc}
\usepackage{lmodern}
\usepackage[margin=2.35cm]{geometry}
\usepackage{amsmath,amssymb,amsthm,mathtools}
\usepackage{array,booktabs,longtable}
\usepackage{tikz}
\usepackage{enumitem}
\usepackage{microtype}
\usepackage{cite}
\usepackage[hidelinks]{hyperref}
\hypersetup{
  pdftitle={Paired Domination in Cubic Bipartite Graphs},
  pdfauthor={Changhong Lu and Qi Wu},
  pdfsubject={Paired domination in cubic bipartite graphs},
  pdfkeywords={paired domination, cubic bipartite graph, perfect matching, Gallai-Edmonds decomposition}
}

\allowdisplaybreaks
\setlist{itemsep=2pt,topsep=4pt}

\newtheorem{theorem}{Theorem}[section]
\newtheorem{lemma}[theorem]{Lemma}
\newtheorem{proposition}[theorem]{Proposition}
\newtheorem{corollary}[theorem]{Corollary}
\theoremstyle{definition}
\newtheorem{definition}[theorem]{Definition}
\newtheorem*{conjecture*}{Conjecture}
\theoremstyle{remark}
\newtheorem{remark}[theorem]{Remark}
\numberwithin{equation}{section}

\newcommand{\gpr}{\gamma_{\mathrm{pr}}}
\newcommand{\defi}{\operatorname{def}}
\newcommand{\OO}{\mathsf O}
\newcommand{\II}{\mathsf I}
\newcommand{\DD}{\mathsf D}
\newcommand{\RR}{\mathsf R}
\newcommand{\sOO}{(\OO,\OO)}
\newcommand{\sOI}{(\OO,\II)}
\newcommand{\sOD}{(\OO,\DD)}
\newcommand{\sOR}{(\OO,\RR)}
\newcommand{\sIO}{(\II,\OO)}
\newcommand{\sII}{(\II,\II)}
\newcommand{\sID}{(\II,\DD)}
\newcommand{\sIR}{(\II,\RR)}
\newcommand{\sDO}{(\DD,\OO)}
\newcommand{\sDI}{(\DD,\II)}
\newcommand{\sDD}{(\DD,\DD)}
\newcommand{\sDR}{(\DD,\RR)}
\newcommand{\sRO}{(\RR,\OO)}
\newcommand{\sRI}{(\RR,\II)}
\newcommand{\sRD}{(\RR,\DD)}
\newcommand{\sRR}{(\RR,\RR)}
\newcommand{\Nout}{N^{+}}
\newcommand{\Nin}{N^{-}}
\newcommand{\dout}{d^{+}}
\newcommand{\din}{d^{-}}
\newcommand{\ClStates}{\Sigma_{\mathrm{cl}}}
\newcommand{\NonOpenStates}{\Sigma_{\mathrm{no}}}

\title{Paired Domination in Cubic Bipartite Graphs}
\author{%
Changhong Lu$^{a}$ and Qi Wu$^{b,*}$\\[4pt]
\small $^{a}$School of Mathematical Sciences, Key Laboratory of MEA (Ministry of Education),\\[-1pt]
\small Shanghai Key Laboratory of PMMP, Nantong Institute for Applied Mathematics\\[-1pt]
\small and Artificial Intelligence, East China Normal University, Shanghai 200241, China\\[2pt]
\small $^{b}$School of Mathematics and Statistics, Jiangsu Normal University,\\[-1pt]
\small Xuzhou, Jiangsu 221116, China}
\date{}

\begin{document}
\maketitle
\begingroup
\renewcommand{\thefootnote}{\fnsymbol{footnote}}
\footnotetext[1]{E-mail addresses:
\href{mailto:chlu@math.ecnu.edu.cn}{\texttt{chlu@math.ecnu.edu.cn}} (C. Lu),
\href{mailto:wuqimath@163.com}{\texttt{wuqimath@163.com}} (Q. Wu, corresponding author).}
\endgroup

\begin{abstract}
A paired dominating set of a graph $G$ is a dominating set $D$ such that
$G[D]$ has a perfect matching. The minimum size of such a set is the paired
domination number $\gpr(G)$. Desormeaux and Henning conjectured that every
cubic bipartite graph $G$ of order $n$ satisfies $\gpr(G)\le n/2$. We prove
the conjecture in the sharp integer form
$\gpr(G)\le 2\lfloor |V(G)|/4\rfloor$ for every finite simple cubic
bipartite graph $G$. The proof combines a
directed contraction along a perfect matching, switching arguments based on
dominator trees, a four-symbol boundary calculus for two-edge cuts, and the
Gallai--Edmonds decomposition. Equality is attained by $K_{3,3}$ when
$|V(G)|\equiv2\pmod4$ and by the cube $Q_3$ when
$|V(G)|\equiv0\pmod4$.
\end{abstract}



\section{Introduction}\label{sec:introduction}

Let $G$ be a graph. We use $V(G)$ and $E(G)$ to
denote the vertex set and edge set of $G$, respectively. A set
$D\subseteq V(G)$ is a \emph{dominating set} if every vertex in
$V(G)\setminus D$ has a neighbor in $D$. A dominating set $D$ is a
\emph{paired dominating set} if $G[D]$ has a perfect matching. The minimum
size of a paired dominating set is the \emph{paired domination number} of
$G$, denoted by $\gpr(G)$.

Haynes and Slater~\cite{HaynesSlater1995,HaynesSlater1998}
introduced paired domination and proved several basic results. Fitzpatrick
and Hartnell~\cite{FitzpatrickHartnell1998} also studied this parameter.
General background on domination can be found in the book of Haynes et
al.~\cite{HaynesHedetniemiSlater1998}. Surveys of paired domination were
given by Desormeaux and Henning~\cite{DesormeauxHenning2014} and by
Desormeaux et al.~\cite{DesormeauxHaynesHenning2020}.

Haynes and Slater~\cite{HaynesSlater1998} proved that $\gpr(G)\le 2n/3$ for
every connected graph $G$ of order $n\ge6$ and minimum degree at least two.
Huang and Shan~\cite{HuangShan2011} later corrected a gap in their proof.
Henning~\cite{Henning2007} studied graphs with large paired domination
number. Further bounds for graphs with minimum degree at least three and
four were obtained by Henning et
al.~\cite{HenningPilsniakTumidajewicz2022} and by Bujt\'as and
Henning~\cite{BujtasHenning2026}, respectively.

For cubic graphs, Chen et al.~\cite{ChenSunXing2007} proved that
$\gpr(G)\le 3n/5$. Goddard and Henning~\cite{GoddardHenning2009} showed that
the Petersen graph is the only connected cubic graph for which equality
holds. Favaron and Henning~\cite{FavaronHenning2004} studied paired
domination in claw-free cubic graphs. Lu et
al.~\cite{LuWangWangWu2019} proved the $4n/7$ bound for connected claw-free
graphs with minimum degree at least three. Sheng and
Lu~\cite{ShengLu2020} obtained the $4n/7$ bound for connected cubic graphs
other than the Petersen graph in a 2020 preprint. Kosari et
al.~\cite{KosariEtAl2022} later published another proof of the same bound.

For cubic bipartite graphs, Desormeaux and
Henning~\cite{DesormeauxHenning2014} proposed the following conjecture.
To the best of our knowledge, no proof for all cubic bipartite graphs has
appeared in the subsequent literature.

\begin{conjecture*}[Desormeaux--Henning]
If $G$ is a cubic bipartite graph of order $n$, then $\gpr(G)\le n/2$.
\end{conjecture*}

Since every paired dominating set has even size, the conjecture is
equivalent to the following integer form for connected graphs.

\begin{theorem}\label{thm:main}
Every finite simple connected cubic bipartite graph $G$ satisfies
\[
 \gpr(G)\le 2\left\lfloor\frac{|V(G)|}{4}\right\rfloor.
\]
Equivalently, $\gpr(G)\le |V(G)|/2$ when $|V(G)|\equiv0\pmod4$, and
$\gpr(G)\le |V(G)|/2-1$ when $|V(G)|\equiv2\pmod4$.
\end{theorem}

The result for a graph with more than one component follows at once.

\begin{corollary}\label{cor:main-all}
Every finite simple cubic bipartite graph $G$ satisfies
$\gpr(G)\le 2\lfloor |V(G)|/4\rfloor$.
\end{corollary}

\begin{proof}
Let $G_1,\ldots,G_k$ be the components of $G$. The paired domination number
is additive over components. By Theorem~\ref{thm:main},
\[
 \gpr(G)=\sum_{i=1}^k\gpr(G_i)
 \le \sum_{i=1}^k 2\left\lfloor\frac{|V(G_i)|}{4}\right\rfloor
 \le 2\left\lfloor\frac{|V(G)|}{4}\right\rfloor.
\]
\end{proof}

The bound is attained in both congruence classes. The two ends of an edge
form a paired dominating set of $K_{3,3}$, and hence
$\gpr(K_{3,3})=2$. In the cube $Q_3$, no edge dominates all eight vertices,
while the four vertices of a face form a paired dominating set. Thus
$\gpr(Q_3)=4$.

The proof has three parts. First, choose a perfect matching of the cubic
bipartite graph and contract its edges. Orient every remaining edge from
one partite class to the other. The resulting directed multigraph has
indegree two and outdegree two at every vertex. We lift a matching of this
multigraph back to the original graph and switch some of its edges until
its endpoints dominate the graph.

Second, consider a two-edge cut. We assign one of four symbols to each
boundary vertex. The symbols record whether the vertex is selected,
whether it is matched inside its side, and whether it needs a selected
neighbor on the other side. A transfer table shows how the two boundary
symbols change when two vertices are added.

Finally, we apply the Gallai--Edmonds decomposition. If the graph obtained
after deleting at most one arc has no perfect matching, its matching
deficiency saves at least two selected vertices. If it has a perfect
matching, a prescribed switch gives the needed boundary state. We then
choose a smallest counterexample and complete the proof.

Several matching-theoretic ingredients used below are standard.
Edge--vertex domination was developed by Lewis~\cite{Lewis2007}, and the
dominating-matching formulation was given by Hedetniemi et
al.~\cite{HedetniemiHedetniemiHaynes2014}. We also use Edmonds' perfect
matching polytope~\cite{Edmonds1965}, Tutte's theorem, maximal barriers, and
the Gallai--Edmonds decomposition; see
Akiyama and Kano~\cite{AkiyamaKano2011} and Lov\'asz and
Plummer~\cite{LovaszPlummer1986}.

The new part of the proof consists of three ingredients. First, we develop
switching lemmas for lifted matchings in legal balanced digraphs. Their
proofs use a Boolean support graph and a counting argument in its dominator
tree. Second, we introduce four boundary symbols and determine the exact
state transfer through a two-vertex extension. Third, we combine a
relative switching lemma with the Gallai--Edmonds decomposition after one
arc is deleted. These ingredients are then joined in a smallest-counterexample
argument.

Section~\ref{sec:contraction-switching} gives the main matching and
switching lemmas. Section~\ref{sec:boundary-transfer} studies the boundary
states at a two-edge cut. Section~\ref{sec:global-proof} proves the main
theorem. The appendix checks the transfer table.

\section{Preliminaries and switching}\label{sec:contraction-switching}

\subsection{Notation}\label{sec:prelim}

All graphs in this paper are finite. Unless stated otherwise, graphs are
simple. A multigraph may have parallel edges, but it has no loops. Its \emph{underlying simple graph} is obtained by replacing each set of
parallel edges by one edge. Parallel edges are treated as distinct
\emph{edge instances} when we discuss matchings or cuts. Adjacency and
domination in a multigraph refer to its underlying simple graph. Paired
domination is defined using this adjacency relation, while a perfect
matching consists of edge instances. Thus parallel edges do not change
domination, but they remain distinct in matchings and cuts.

The order of a graph $G$ is $|V(G)|$. For $v\in V(G)$, let $N_G(v)$ and
$d_G(v)$ denote its open neighborhood and degree. A graph is \emph{cubic}
if every vertex has degree three. For a nonempty proper set
$U\subset V(G)$, let $\delta_G(U)$ be the set of edge instances with
exactly one end in $U$. This set is the \emph{edge cut} determined by
$U$. An edge cut of size two is a \emph{two-edge cut}.
An edge is a \emph{bridge} if its deletion increases the number of
components. A graph is $k$-edge-connected if
$|\delta_G(U)|\ge k$ for every nonempty proper set $U\subset V(G)$.

If $\delta_G(U)$ is an edge cut, then $G[U]$ and
$G[V(G)\setminus U]$ are called the two \emph{sides} of the cut. A vertex
of a side is a \emph{boundary vertex} if it is incident with an edge of
the cut. If $A,B\subseteq V(G)$ are disjoint, let
$E_G(A,B)$ be the set of edge instances with one end in $A$ and the other
in $B$, and put $e_G(A,B)=|E_G(A,B)|$. We also write
$E_G(A)=E(G[A])$. For a graph $J$, let $o(J)$ be the number of odd
components of $J$.

A matching is a set of pairwise disjoint edge instances. A
\emph{maximum matching} has the largest possible number of edges. For a
matching $K$, let $V(K)$ be the set of its endpoints. A matching is \emph{perfect}
if it covers every vertex. It is \emph{near-perfect} if it covers all but
one vertex; the uncovered vertex is called the \emph{exposed vertex}. A
matching $K$ is \emph{dominating} if $V(K)$ dominates the graph. A graph
is \emph{factor-critical} if deleting any vertex leaves a graph with a
perfect matching. For two edge sets $A$ and $B$, their symmetric
difference is denoted by $A\mathbin\triangle B$.

We use $\gpr(G)$ for the paired domination number defined in the
Introduction.

The following elementary degree-counting observation will be used
repeatedly.

\begin{lemma}\label{lem:cubic-bipartite-bridgeless}
Every connected cubic bipartite multigraph is bridgeless.
\end{lemma}

\begin{proof}
Let $G=(X,Y;E)$ be connected and cubic, and suppose that $e$ is a bridge.
For one component $U$ of $G-e$, let $d_X$ and $d_Y$ be the numbers of cut
edges whose end in $U$ lies in $X$ and $Y$, respectively. Since the cut
consists only of $e$, we have $d_X-d_Y\in\{1,-1\}$. Counting the ends of
edges inside $U$ from the two partite classes gives
$3(|X\cap U|-|Y\cap U|)=d_X-d_Y$, which is impossible. Thus $G$ has no
bridge.
\end{proof}

For an arc $u\to v$, the vertex $u$ is its \emph{tail} and $v$ is its
\emph{head}. The underlying multigraph of a directed multigraph is obtained
by forgetting all arc directions. For a directed multigraph $\vec H$, the
out- and in-neighborhoods of a vertex $v$ are $\Nout_{\vec H}(v)$ and
$\Nin_{\vec H}(v)$. The corresponding degrees are
$\dout_{\vec H}(v)$ and $\din_{\vec H}(v)$. When an undirected edge
instance comes from an arc, we keep the direction of that arc in the
notation.

\subsection{Matchings and contractions}\label{sec:contraction}

Edge--vertex domination was studied in detail by Lewis~\cite{Lewis2007}.
Hedetniemi, Hedetniemi, and Haynes~\cite{HedetniemiHedetniemiHaynes2014}
proved its equivalence with paired domination. In the matching form used
here, their result is the following. For completeness, we include a short proof.

\begin{proposition}[Hedetniemi--Hedetniemi--Haynes]
\label{prop:matching-form}
For every graph $G$ without isolated vertices,
\[
 \gpr(G)=2\min\bigl\{|K|:K\text{ is a matching and }V(K)\text{ dominates }G\bigr\}.
\]
\end{proposition}

\begin{proof}
Let $D$ be a paired dominating set, and let $K$ be a perfect matching of
$G[D]$. Then $V(K)=D$, so $K$ is a dominating matching and
$|D|=2|K|$. Conversely, if $K$ is a matching and $V(K)$ dominates $G$,
then $V(K)$ is a paired dominating set whose induced subgraph contains
$K$ as a perfect matching. Taking the minimum in both directions proves
the equality.
\end{proof}

By Hall's theorem, every $r$-regular bipartite graph with $r\ge1$ has a
perfect matching. Repeatedly deleting a perfect matching gives a
decomposition into perfect matchings. We use this fact below.

\begin{definition}\label{def:legal}
A directed multigraph $\vec H$ is \emph{balanced $2$-in--$2$-out} if every
vertex has indegree two and outdegree two. It is \emph{legal} if it is
loopless, the two out-neighbors of each vertex are distinct, and the two
in-neighbors of each vertex are distinct. A directed multigraph with both
properties is called a \emph{legal balanced digraph}.

The \emph{canonical bipartite lift} $L(\vec H)$ has partite sets
$X=\{x_v:v\in V(\vec H)\}$ and $Y=\{y_v:v\in V(\vec H)\}$. It
contains the fixed edge $x_v y_v$ for every $v$ and the edge $x_u y_v$
for every arc $u\to v$. If $P$ is a matching in the underlying
multigraph of $\vec H$, orient each edge instance of $P$ as it occurs in
$\vec H$. The corresponding edges $x_u y_v$ form the \emph{standard lift}
of $P$.

Let $R$ be a set containing at most one arc of $\vec H$, and put
$\vec Q=\vec H-R$. The \emph{canonical lift associated with
$(\vec H,\vec Q)$} is
$L(\vec H,\vec Q)=L(\vec H)-\{x_u y_v:u\to v\in R\}$. If $R$ contains
one arc, we also call $L(\vec H,\vec Q)$ the
\emph{edge-deleted canonical lift}.

Conversely, let $G=(X,Y;E)$ be a simple cubic bipartite graph and let $M$
be a perfect matching. Write $M=\{x_v y_v:v\in I\}$, and contract every
edge of $M$. Each edge $x_u y_v\in E(G)\setminus M$ becomes
the arc $u\to v$. The resulting digraph is the \emph{legal contraction}
of $(G,M)$.
\end{definition}

The conditions in Definition~\ref{def:legal} imply that $L(\vec H)$ is a
simple cubic bipartite graph. Conversely, the legal contraction of
$(G,M)$ is balanced $2$-in--$2$-out, its underlying multigraph is
$4$-regular, and $G=L(\vec H)$.

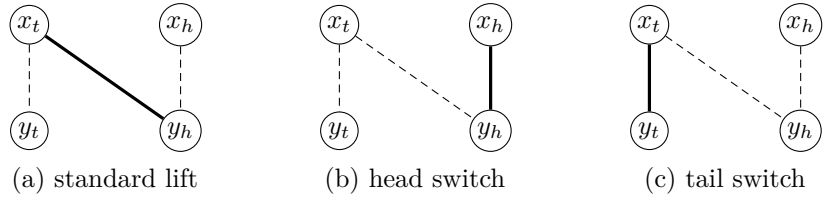
\begin{figure}[htbp]
\centering
\begin{tikzpicture}[
  vertex/.style={circle,draw,inner sep=1.3pt,font=\small},
  selected/.style={very thick},
  available/.style={densely dashed},
  every node/.style={font=\small}
]
% Standard lift.
\node[vertex] (axt) at (0,0.7) {$x_t$};
\node[vertex] (ayt) at (0,-0.7) {$y_t$};
\node[vertex] (axh) at (2,0.7) {$x_h$};
\node[vertex] (ayh) at (2,-0.7) {$y_h$};
\draw[available] (axt)--(ayt);
\draw[available] (axh)--(ayh);
\draw[selected] (axt)--(ayh);
\node at (1,-1.35) {(a) standard lift};

% Head switch.
\node[vertex] (bxt) at (4.1,0.7) {$x_t$};
\node[vertex] (byt) at (4.1,-0.7) {$y_t$};
\node[vertex] (bxh) at (6.1,0.7) {$x_h$};
\node[vertex] (byh) at (6.1,-0.7) {$y_h$};
\draw[available] (bxt)--(byt);
\draw[available] (bxt)--(byh);
\draw[selected] (bxh)--(byh);
\node at (5.1,-1.35) {(b) head switch};

% Tail switch.
\node[vertex] (cxt) at (8.2,0.7) {$x_t$};
\node[vertex] (cyt) at (8.2,-0.7) {$y_t$};
\node[vertex] (cxh) at (10.2,0.7) {$x_h$};
\node[vertex] (cyh) at (10.2,-0.7) {$y_h$};
\draw[selected] (cxt)--(cyt);
\draw[available] (cxt)--(cyh);
\draw[available] (cxh)--(cyh);
\node at (9.2,-1.35) {(c) tail switch};
\end{tikzpicture}
\caption{The three local choices associated with a matched arc $t\to h$.
A thick edge is selected, while dashed edges are available but not selected.}
\label{fig:lift-switches}
\end{figure}

\begin{lemma}\label{lem:cut}
Let $G=(X,Y;E)$ be a finite simple cubic bipartite graph, let $M$ be a
perfect matching of $G$, and let $\vec H$ be the legal contraction of
$(G,M)$. If $G$ is three-edge-connected, then the underlying multigraph
$H$ of $\vec H$ is four-edge-connected.
\end{lemma}

\begin{proof}
A vertex set of $H$ corresponds to a union of edges of $M$, and its edge cut
is the same set of edges outside $M$ as the corresponding cut of $G$. Every cut
of a $4$-regular multigraph has even size. Since $G$ has no cut of
size one or two, every proper nonempty cut of $H$ has size at least four.
\end{proof}

The next lemma is an immediate specialization of Edmonds' description of
the perfect matching polytope~\cite{Edmonds1965}; see also Lov\'asz and
Plummer~\cite{LovaszPlummer1986}. We include the short verification because
the probabilistic form is used repeatedly.

\begin{lemma}[Edmonds]\label{lem:uniform-even}
Let $H$ be an even-order loopless $4$-regular $4$-edge-connected multigraph.
There is a probability distribution on the perfect matchings of $H$ such
that every edge instance belongs to a random perfect matching with
probability $1/4$. In particular, every edge instance belongs to some
perfect matching.
\end{lemma}

\begin{proof}
By Edmonds' perfect matching polytope theorem~\cite{Edmonds1965}, it is
enough to verify that the constant vector $x_e=1/4$ satisfies the vertex
and odd-cut constraints. We have $x(\delta_H(v))=1$ for every vertex $v$.
If $U$ has odd order, then $|\delta_H(U)|\ge4$, and hence
$x(\delta_H(U))\ge1$. Thus $x$ is a convex combination of incidence
vectors of perfect matchings. The coefficients give the required
distribution.
\end{proof}

\begin{remark}[Matchings in multigraphs]\label{rem:multigraph-matching}
In Edmonds' theorem, we use one variable for each edge instance. For the
Gallai--Edmonds decomposition, each parallel class may be replaced by one
edge. This preserves adjacency, components, matching size, the possible
sets of covered vertices, and factor-criticality. After the decomposition
is applied, we choose the needed edge instances and restore their
directions.
\end{remark}

\begin{theorem}\label{thm:even-contraction}
Let $G$ be a finite simple three-edge-connected cubic bipartite graph. If
$|V(G)|\equiv0\pmod4$, then $\gpr(G)\le |V(G)|/2$.
\end{theorem}

\begin{proof}
Choose a perfect matching $M$ of $G$, let $\vec H$ be the legal contraction
of $(G,M)$, and let $H$ be its underlying multigraph. Set $m=|V(H)|=|V(G)|/2$. By Lemma~\ref{lem:cut}, the multigraph $H$
is $4$-edge-connected. Since $m$ is even,
Lemma~\ref{lem:uniform-even} gives a perfect matching
$P$ of $H$. Let $K$ be the standard lift of $P$. Then $|K|=m/2$.
Exactly one end of every fixed edge $x_v y_v\in M$ belongs to $V(K)$, so
the other end is dominated through that fixed edge. Hence $K$ is a
dominating matching with $|V(K)|=m=|V(G)|/2$. The result follows from Proposition~\ref{prop:matching-form}.
\end{proof}

Every cut of a $4$-regular multigraph has even size. Hence, in this
setting, $3$-edge-connectivity and $4$-edge-connectivity are equivalent.
We record the following standard consequence of Tutte's theorem in a form
that applies directly to multigraphs.

\begin{lemma}\label{lem:factor-critical}
Let $H$ be an odd-order $4$-regular $4$-edge-connected multigraph. Then
$H$ is factor-critical.
\end{lemma}

\begin{proof}
Fix $v\in V(H)$. We verify Tutte's condition for $H-v$. Let
$S\subseteq V(H)\setminus\{v\}$, and let $q=o(H-v-S)$. Every odd
component of $H-v-S$ is a nonempty proper vertex set of $H$, and hence it
has at least four boundary edge instances. All these boundary edges end
in $S\cup\{v\}$. Since every vertex of $S\cup\{v\}$ has degree four,
\[
 4q\le 4|S|+4,
\]
so $q\le |S|+1$.

The graph $H-v-S$ has order congruent to $|S|$ modulo two. The parity of
its order is also the parity of its number of odd components, and therefore
$q\equiv |S|\pmod2$. The inequality $q\le |S|+1$ now gives
$q\le |S|$. By Tutte's theorem, $H-v$ has a perfect matching. Since $v$
was arbitrary, $H$ is factor-critical.
\end{proof}

\subsection{A counting lemma for the dominator tree}\label{sec:dominator}

We use the dominator terminology of Lengauer and
Tarjan~\cite{LengauerTarjan1979}.

\begin{definition}\label{def:dominator-tree}
A \emph{rooted flow graph} is a digraph with a root $s$ from which every
vertex is reachable. A vertex $u$ \emph{dominates} a vertex $v$ if every
directed $s$--$v$ path contains $u$. It \emph{strictly dominates} $v$ if
$u\ne v$. For $v\ne s$, the closest strict dominator of $v$ is its
\emph{immediate dominator}. The immediate-dominator relation forms the
\emph{dominator tree}, rooted at $s$.
\end{definition}

\begin{lemma}\label{lem:dominator-leaves}
Let $Q$ be a finite acyclic rooted flow graph with root $s$, and assume that
$Q$ has at least one nonroot vertex. Suppose that distinct arcs entering
any vertex have distinct tails. If every nonroot vertex has indegree one or
two, and $J$ nonroot vertices have indegree two, then the dominator tree has
at least $J+1$ nonroot leaves.
\end{lemma}

\begin{proof}
Take a topological ordering beginning with $s$, and add the nonroot vertices
in that order. A later vertex cannot lie on a directed path to an earlier
vertex, so this process does not change domination among the vertices
already present.

Let $v$ be the next vertex, and let $P(v)$ be its set of predecessors. A
vertex $d\ne v$ dominates $v$ exactly when it dominates every vertex of
$P(v)$. Indeed, every directed $s$--$v$ path ends with an arc $u\to v$
for some $u\in P(v)$, and every directed $s$--$u$ path extends to an
$s$--$v$ path. Thus $\operatorname{Dom}(v)\setminus\{v\}
=\bigcap_{u\in P(v)}\operatorname{Dom}(u)$. Hence the immediate dominator
of $v$ is the lowest common ancestor of its predecessors in the current
dominator tree.

If $v$ has one predecessor, attaching $v$ does not reduce the number of
nonroot leaves. If $v$ has two predecessors, their tails are distinct by
assumption. Their lowest common ancestor is already a nonleaf, so the new
vertex adds one leaf and removes none. The first nonroot vertex contributes
one leaf, and every vertex of indegree two contributes one further leaf.
Thus the dominator tree has at least $J+1$ nonroot leaves.
\end{proof}

\subsection{Switching with one exposed vertex}\label{sec:one-sided}

\begin{definition}\label{def:directional-certificate}
Let $\vec H$ be a legal balanced digraph of odd order $m=2r+1$, and let
$P$ be a near-perfect matching of its underlying multigraph, exposed at $z$.
Write every edge instance of $P$ in its direction in $\vec H$ as
$P=\{t_i\to h_i:1\le i\le r\}$. Put $T_P=\{t_1,\ldots,t_r\}$ and $H_P=\{h_1,\ldots,h_r\}$. The pair
$(P,z)$ has a \emph{positive certificate} if
$\Nout_{\vec H}(z)\cap H_P\ne\varnothing$. It has a \emph{negative
certificate} if $\Nin_{\vec H}(z)\cap T_P\ne\varnothing$.
\end{definition}

\begin{remark}\label{rem:legal-necessary}
The assumption that the two in-neighbors and the two out-neighbors are
distinct is necessary. Take three vertices indexed modulo three and add
two parallel arcs $i\to i+1$ for each $i$. The digraph is loopless and
balanced $2$-in--$2$-out. Its underlying doubled triangle is $4$-regular
and $4$-edge-connected. Every near-perfect matching consists of one arc
$i\to i+1$ and exposes $i+2$, but it has neither certificate from
Definition~\ref{def:directional-certificate}.
\end{remark}

\begin{definition}
\label{def:switching-support}
For a matched arc $t_i\to h_i$, a head-switching variable $\xi_i$ has the
following meaning. The value $0$ keeps the lifted edge
$x_{t_i} y_{h_i}$, while the value $1$ replaces it by the fixed edge
$x_{h_i} y_{h_i}$. For a vertex $w$ covered by the matching, put
\[
 \lambda(w)=
 \begin{cases}
  \neg\xi_j,&w=t_j,\\
  \xi_j,&w=h_j.
 \end{cases}
\]
A tail clause has the form
$C_i=\neg\xi_i\vee L_{i,1}\vee L_{i,2}$. The first literal $\neg\xi_i$ is the \emph{guard}. The last two entries
are the \emph{support positions}. A positive or negative support
occurrence means a literal in one of these two positions. The guard is
never counted as a support occurrence. A clause of the form
$\bigvee_{j\in S_0}\xi_j$ placed before the tail clauses is the
\emph{source clause}. Its literals are represented by source arcs
$s\to j$ and are not counted in $E_+$.
\end{definition}

We use the following rule in the switching arguments.
In a tail clause $C_i$, an arc $h_j\to t_i$ contributes the positive
support $\xi_j$, while an arc $t_j\to t_i$ contributes the negative support
$\neg\xi_j$. A source arc $s\to j$ represents a prescribed positive
support $\xi_j$ and is not counted as an ordinary positive support
occurrence. In the relative setting used later, an arc from an index whose
fixed edge has already been selected contributes the constant $1$, an arc
from the exposed index contributes the constant $0$, and a deleted arc
contributes no support position. Guards are never counted as support
occurrences.

\begin{lemma}\label{lem:support-obstruction}
Let $I$ be a finite index set, let $S_0\subseteq I$ be nonempty, and
consider a Boolean formula
\[
 \Psi=\left(\bigvee_{j\in S_0}\xi_j\right)
 \wedge\bigwedge_{i\in I}
 \left(\neg\xi_i\vee\bigvee_{\ell\in\Lambda_i}\ell\right),
\]
where every member of $\Lambda_i$ is one of
$0,1,\xi_j,\neg\xi_j$. The first literal $\neg\xi_i$ is the guard and the
members of $\Lambda_i$ are the support positions.

Build a directed support graph with a source $s$: add $s\to j$ for
$j\in S_0$, and add $i\to j$ for every positive support occurrence
$\xi_j\in\Lambda_i$. A negative support occurrence
$\neg\xi_j\in\Lambda_i$ is denoted by $i\to\overline j$. Let $Q$ be the
subgraph induced by $s$ and the variables reachable from $s$ by positive
arcs. A tail clause indexed by $i$ is called \emph{reachable} if $i$ is a
vertex of $Q$. If $\Psi$ is not satisfiable, then the following statements
hold.
\begin{enumerate}[label=\textnormal{(\roman*)}]
\item The positive support graph $Q$ is acyclic.
\item No reachable tail clause contains the constant support $1$.
\item If $i\to\overline j$ is a negative support with $i$ reachable, then
$j$ lies on every directed $s$--$i$ path. In particular, $j$ is reachable
and dominates $i$ in $Q$; if $j\ne i$, then $j$ strictly dominates $i$ and
is a nonleaf of the dominator tree.
\item If each variable occurs as a negative support in at most one support
position, then the negative targets obtained from reachable tail clauses are
pairwise distinct.
\end{enumerate}
\end{lemma}

\begin{proof}
Suppose first that $Q$ contains a directed cycle. Choose a directed path
from $s$ to the cycle that first meets the cycle at its last vertex, and
set to one exactly the variables on this path and cycle. The source clause
is true. Every true variable on the path or cycle has a positive support
at the next true variable, while every false variable satisfies its guard.
Thus the assignment satisfies $\Psi$.

If a reachable tail clause contains the support $1$, set to one exactly the
variables on a directed path from $s$ to that clause. The last clause is
satisfied by the constant support and every earlier true clause is
supported by the next variable on the path. This also satisfies $\Psi$.

Finally, suppose that $i\to\overline j$ is a negative support and that some
directed $s$--$i$ path avoids $j$. Setting to one exactly the variables on
that path satisfies the last clause by $\neg\xi_j$ and every earlier true
clause by the next positive support. Thus, if $\Psi$ is not satisfiable, none of the three cases above can
occur. This proves the first three statements. The last statement follows
from the assumption that each variable appears as a negative support in at
most one position.
\end{proof}

\begin{theorem}\label{thm:one-sided-switch}
Let $\vec H$ be a legal balanced digraph of odd order
$m=2r+1$, and let $P$ be a near-perfect matching of its underlying
multigraph, exposed at $z$. If $(P,z)$ has a positive certificate or a
negative certificate, then $L(\vec H)$ has a dominating matching with $r$
edges. Consequently, $L(\vec H)$ has a paired dominating set of size
$2r=m-1$. Moreover, the dominating matching can be chosen to use only
fixed edges and lifts of edges of $P$.
\end{theorem}

\begin{proof}
We prove the positive case. The negative case follows by reversing all
arcs and interchanging the two partite classes.

\medskip
\noindent\emph{Step 1: the switching formula.}
Write the matched edge instances of $P$ in their directions as
$P=\{t_i\to h_i:1\le i\le r\}$. For each $i$, choose a variable
$\xi_i\in\{0,1\}$. If $\xi_i=0$, select the lifted edge
$x_{t_i} y_{h_i}$; if $\xi_i=1$, replace it by the fixed edge
$x_{h_i} y_{h_i}$, as in Figure~\ref{fig:lift-switches}. The selected
edges form a matching $K_\xi$ of size $r$.

The selected $Y$-indices are always $H_P$, and the selected $X$-indices
are
\[
 X_\xi=\{t_i:\xi_i=0\}\cup\{h_i:\xi_i=1\}.
\]
No lift vertex is selected at an index in
$Z_\xi=\{z\}\cup\{t_i:\xi_i=1\}$. If $\xi_i=1$, then
$x_{t_i}$ is dominated by the selected vertex $y_{h_i}$ through the edge
$x_{t_i} y_{h_i}$, so only $y_{t_i}$ still needs a selected $X$-neighbor.
The positive certificate gives an arc from $z$ to a matched head, and
therefore the corresponding selected $Y$-vertex dominates $x_z$. Thus
only $y_z$ remains to be dominated at the exposed index.

If an in-neighbor of $z$ lies in $T_P$, the all-zero assignment selects
that tail on the $X$-side and works. We may therefore write
$\Nin_{\vec H}(z)=\{h_p,h_q\}$, where $p\ne q$. Extend the literal map
from Definition~\ref{def:switching-support} by putting $\lambda(z)=0$.
If $u_i$ and $v_i$ are the two in-neighbors of $t_i$, the remaining
domination conditions are exactly
\begin{equation}
 \Phi=(\xi_p\vee\xi_q)\wedge
 \bigwedge_{i=1}^{r}
 \bigl(\neg\xi_i\vee\lambda(u_i)\vee\lambda(v_i)\bigr).
 \label{eq:one-sided-formula}
\end{equation}
It remains to prove that the formula in~\eqref{eq:one-sided-formula} is satisfiable.

\medskip
\noindent\emph{Step 2: consequences of an obstruction.}
Build the support graph from Lemma~\ref{lem:support-obstruction}, with
source arcs $s\to p$ and $s\to q$. Let $Q$ be its reachable positive
part, and let $\nu$ be the number of reachable variables. Suppose that
$\Phi$ is not satisfiable. Then $Q$ is acyclic, no reachable clause
contains the constant support $1$, and every negative target is reachable
and dominates the index of its clause. This domination is strict: a
negative support $\neg\xi_i$ in clause $i$ would come from a loop
$t_i\to t_i$.

\medskip
\noindent\emph{Step 3: counting positive and negative supports.}
Let $E_+$ and $E_-$ be the numbers of positive and negative support
occurrences in the reachable tail clauses. A positive support arc entering
variable $j$ records an arc leaving $h_j$: a source arc records
$h_j\to z$, and an ordinary support arc records $h_j\to t_i$. The two
out-neighbors of $h_j$ are distinct. Hence distinct arcs entering a
variable have distinct tails, and every reachable variable has indegree
one or two in $Q$. If $J$ of the reachable variables have indegree two,
then counting all arcs entering variable vertices gives
\begin{equation}
 E_++2=\nu+J.
 \label{eq:one-sided-indegree}
\end{equation}

Each reachable tail clause has two support positions. A zero support can
only arise from an arc $z\to t_i$. At least one out-arc of $z$ gives the
positive certificate and ends at a matched head, so at most one support
position in the reachable clauses is zero. Every positive target is
reachable by definition, and every negative target is reachable by
Lemma~\ref{lem:support-obstruction}. Therefore
\begin{equation}
 E_++E_-\ge 2\nu-1.
 \label{eq:one-sided-supports}
\end{equation}
Combining \eqref{eq:one-sided-indegree} and
\eqref{eq:one-sided-supports} yields
\begin{equation}
 E_-\ge \nu+1-J.
 \label{eq:one-sided-negative}
\end{equation}

\medskip
\noindent\emph{Step 4: the dominator-tree contradiction.}
A negative support with target $j$ comes from an arc $t_j\to t_i$. The
matched arc $t_j\to h_j$ already uses one of the two out-arcs of $t_j$.
Thus each variable occurs as a negative target at most once. The $E_-$
negative targets are therefore distinct nonleaves of the dominator tree.
By Lemma~\ref{lem:dominator-leaves}, this tree has at least $J+1$
variable leaves and hence at most $\nu-J-1$ nonleaf variables. This
contradicts \eqref{eq:one-sided-negative}. Hence $\Phi$ is satisfiable.

Choose a satisfying assignment. The clauses dominate both lift vertices
at every index in $Z_\xi$. At every remaining index, at least one lift
vertex is selected. Whenever exactly one lift vertex at an index is
selected, the other is dominated through the fixed edge. Thus $K_\xi$ is
a dominating matching with $r$ edges. Its endpoint set has size
$|V(K_\xi)|=2r=m-1$, and Proposition~\ref{prop:matching-form} completes
the proof.
\end{proof}

\subsection{A near-perfect matching with a suitable direction}\label{sec:directional}

We use the standard deficiency and barrier terminology from matching
theory; see Akiyama and Kano~\cite{AkiyamaKano2011} or Lov\'asz and
Plummer~\cite{LovaszPlummer1986}.

\begin{definition}\label{def:barrier}
For a loopless multigraph $J$, put
$\defi(J)=\max\{o(J-S)-|S|:S\subseteq V(J)\}$. This number is the
\emph{deficiency} of $J$. A set $S$ attaining the maximum is a
\emph{barrier}. It is a \emph{maximal barrier} if it is maximal under
inclusion.
\end{definition}

We record the standard maximal-barrier property in a form that also
applies to loopless multigraphs. A direct proof is included because the
result is used in the directional matching argument below.

\begin{lemma}[Maximal-barrier property]\label{lem:maximal-barrier}
Let $J$ be a loopless multigraph, and let $S$ be an inclusion-maximal
barrier of $J$. Every component of $J-S$ is factor-critical.
\end{lemma}

\begin{proof}
Put
\[
 d=o(J-S)-|S|=\defi(J).
\]
We first show that every component of $J-S$ has odd order. Suppose that
$C$ is an even component of $J-S$, and choose $v\in V(C)$. The graph
$C-v$ has odd order and therefore has at least one odd component. Hence
\[
 o\bigl(J-(S\cup\{v\})\bigr)-|S\cup\{v\}|
 \ge o(J-S)+1-|S|-1=d.
\]
By the definition of deficiency, equality holds, so $S\cup\{v\}$ is a
barrier. This contradicts the inclusion-maximality of $S$. Thus every
component of $J-S$ has odd order.

Let $C$ be a component of $J-S$, and fix $v\in V(C)$. Suppose that
$C-v$ has no perfect matching. By Tutte's theorem, there is a set
$T\subseteq V(C)\setminus\{v\}$ such that
\[
 o(C-v-T)>|T|.
\]
Since $|V(C)|$ is odd, the order of $C-v-T$ has the same parity as
$|T|$. The parity of a graph order is the parity of its number of odd
components. Hence $o(C-v-T)\equiv |T|\pmod2$, and consequently
\[
 o(C-v-T)\ge |T|+2.
\]
Set $S'=S\cup\{v\}\cup T$. Every component of $J-S$ other than $C$
remains an odd component of $J-S'$. It follows that
\[
 \begin{aligned}
 o(J-S')-|S'|
 &\ge o(J-S)-1+(|T|+2)-(|S|+|T|+1)\\
 &=d.
 \end{aligned}
\]
Again equality must hold, so $S'$ is a barrier properly containing $S$,
a contradiction. Hence $C-v$ has a perfect matching for every
$v\in V(C)$, and $C$ is factor-critical.
\end{proof}

\begin{lemma}\label{lem:directional-certificate}
Let $\vec H$ be a legal balanced digraph whose underlying
multigraph $H$ has odd order and is $4$-regular and $4$-edge-connected.
For every prescribed arc $a\to v$, the multigraph $H$ has a near-perfect
matching that avoids this arc and has a positive certificate or a negative
certificate at its exposed vertex.
\end{lemma}

\begin{proof}
Fix the arc $a\to v$, and let $b\to v$ be the other arc entering $v$.
Since the digraph is legal, $a\ne b$. Put $F=\{a,v,b\}$.

First suppose that $H-F$ has a perfect matching $Q$. Then
$P=Q\cup\{b\to v\}$ is a near-perfect matching exposed at $a$, and it
avoids $a\to v$. The
arc $a\to v$ leaves the exposed vertex and enters a matched head. Thus
$(P,a)$ has a positive certificate.

Now suppose that $H-F$ has no perfect matching. Let $S$ be an
inclusion-maximal barrier of $H-F$. Write $s=|S|$, and let $q$ be the
number of odd components of $H-F-S$. By Tutte's theorem,
$\defi(H-F)>0$. Moreover, $H-F$ has even order, so
$o(H-F-S)-|S|$ is even for every $S$. Therefore
$q-s=\defi(H-F)\ge2$. Let $\rho=|E(H[F])|$. The two edges $av$ and $bv$ are distinct, so
$\rho\ge2$ and $|\delta_H(F)|=12-2\rho\le8$. Every odd component of $H-F-S$ has at least four boundary edges, all ending
in $S\cup F$. Hence $4s+8\le4q\le4s+|\delta_H(F)|\le4s+8$.
All three inequalities are equalities. Therefore $q=s+2$ and $\rho=2$,
and each odd component has exactly four boundary edges.

Let $O$ and $E$ be the unions of the odd and even components of $H-F-S$.
The total boundary of the odd components is $|E_H(O,S\cup F)|$, so
$|E_H(O,S\cup F)|=4q=4s+|\delta_H(F)|$. Expanding the degree sum over
$S$ and the boundary of $F$, and using that distinct components of
$H-F-S$ have no edges between them, gives
\[
 \begin{aligned}
 4s+|\delta_H(F)|-|E_H(O,S\cup F)|
 ={}&2|E_H(S)|+2|E_H(S,F)|\\
 &+|E_H(S,E)|+|E_H(F,E)|.
 \end{aligned}
\]
Every term on the right is zero. Thus $S$ is independent,
$E_H(S,F)=\varnothing$, and no even component has an edge to $S\cup F$.
Since $H$ is connected, there are no even components. By
Lemma~\ref{lem:maximal-barrier}, every component of $H-F-S$ is
factor-critical. Let $\mathcal Q$ be the family of these components.

Put $L=S\cup\{a,b\}$. Since $\rho=2$, the only edges in $H[F]$ are
$av$ and $bv$. Hence $a$ and $b$ are nonadjacent. These equalities also
show that $L$ is independent and $|L|=s+2=q=|\mathcal Q|$. Form a bipartite multigraph $J$ with parts $L$ and $\mathcal Q$. Each edge
instance of $H$ from $\ell\in L$ to a component $Q\in\mathcal Q$ gives one
edge of $J$, labeled by that edge instance.

By Lemma~\ref{lem:factor-critical}, $H-v$ has a perfect matching $M_v$.
Every component in $\mathcal Q$ has odd order, so $M_v$ uses an edge from
that component to $L$. Since $|L|=|\mathcal Q|$, it uses exactly one such
edge at each vertex of $L$ and at each component. The corresponding edges
form a perfect matching of $J$.

Let $a\to\alpha$ be the out-arc of $a$ different from $a\to v$. The
vertex $\alpha$ lies in some component $Q_0\in\mathcal Q$. Let
$e_{a\alpha}$ be the corresponding edge of $J$. We claim that $J$ has a
perfect matching containing $e_{a\alpha}$.

Suppose not. Then $J-a-Q_0$ has no perfect matching. By Hall's theorem,
there is a nonempty set $U\subseteq L\setminus\{a\}$ such that
$|N_{J-a-Q_0}(U)|\le |U|-1$. Since $J$ has a perfect matching,
$|N_J(U)|\ge |U|$. Deleting $Q_0$ changes this neighborhood by at most
one vertex. Therefore, for $N=N_J(U)$, we have $|N|=|U|$ and $Q_0\in N$.
Let $\kappa=1$ when $b\in U$, and let $\kappa=0$ otherwise. Each vertex
of $U\cap S$ sends four edges to $N$, while $b$ sends three. Thus
$e_J(U,N)=4|U|-\kappa$. Put
$X_U=U\cup\bigcup_{Q\in N}V(Q)$. We obtain
\[
 \begin{aligned}
 |\delta_H(X_U)|
 &=4|U|+4|N|-2e_J(U,N)\\
 &=2\kappa\le2.
 \end{aligned}
\]
The set $X_U$ is nonempty and proper because $v\notin X_U$. This
contradicts four-edge-connectivity and proves the claim.

Choose a perfect matching $M_J$ of $J$ containing $e_{a\alpha}$. For each
$Q\in\mathcal Q$, let $\ell_Q u_Q$ be the edge of $H$ labeling the edge of
$M_J$ incident with $Q$, where $\ell_Q\in L$ and $u_Q\in V(Q)$. Since
$Q$ is factor-critical, choose a perfect matching $M_Q$ of $Q-u_Q$. Then
$P=\{\ell_Q u_Q:Q\in\mathcal Q\}\cup\bigcup_{Q\in\mathcal Q}M_Q$ is a
perfect matching of $H-v$. It contains the arc $a\to\alpha$ and
avoids $a\to v$. Thus $a\in T_P$, and the arc $a\to v$ gives a negative
certificate at the exposed vertex $v$.
\end{proof}

\begin{theorem}\label{thm:intro-three-edge}
Let $G$ be a finite simple three-edge-connected cubic bipartite graph of
order $n$. Then $\gpr(G)\le n/2$. If $n\equiv2\pmod4$, then
$\gpr(G)\le n/2-1$.
\end{theorem}

\begin{proof}
If $n\equiv0\pmod4$, apply Theorem~\ref{thm:even-contraction}. Suppose
$n\equiv2\pmod4$ and put $m=n/2$. Choose a perfect matching $M$ of $G$,
and let $\vec H$ be the legal contraction of $(G,M)$. Its underlying
multigraph is $4$-edge-connected by Lemma~\ref{lem:cut}. Since $m$ is odd,
Lemma~\ref{lem:directional-certificate}, applied to any arc of $\vec H$,
gives a near-perfect matching with a directional certificate.
Theorem~\ref{thm:one-sided-switch} then gives a paired dominating set of
size $m-1=n/2-1$.
\end{proof}

\subsection{Switching at a given index}\label{sec:prescribed}

In the next theorem, the symbols
$x_v,y_v$ are as defined in Definition~\ref{def:legal}.

\begin{theorem}\label{thm:prescribed-switch}
Let $\vec H$ be a legal balanced digraph of even order, and let
$P=\{t_i\to h_i:1\le i\le r\}$ be a perfect matching of its underlying
multigraph. For every index $k$, there is a set
$I\subseteq\{1,\ldots,r\}$ with $k\in I$ such that the edges
$x_{t_i} y_{h_i}$ for $i\notin I$ and $x_{h_i} y_{h_i}$ for $i\in I$
form a dominating matching of $L(\vec H)$. Symmetrically, there is a set
$I\ni k$ such that the edges $x_{t_i} y_{h_i}$ for $i\notin I$ and
$x_{t_i} y_{t_i}$ for $i\in I$ form a dominating matching of
$L(\vec H)$.
\end{theorem}

\begin{proof}
We prove the first statement. Use the head-switching variables and the
literal map from the proof of Theorem~\ref{thm:one-sided-switch}. Since
$P$ is perfect, every in-neighbor of every $t_i$ is a matched tail or a
matched head, so no support is the constant $0$. Prescribing the switch at
$k$ gives the formula
\begin{equation}
 \xi_k\wedge
 \bigwedge_{i=1}^{r}
 \bigl(\neg\xi_i\vee\lambda(u_i)\vee\lambda(v_i)\bigr),
 \label{eq:prescribed-formula}
\end{equation}
where $u_i$ and $v_i$ are the two in-neighbors of $t_i$.

Build the support graph with the single source arc $s\to k$, and let $Q$
be its reachable positive part. Suppose that
\eqref{eq:prescribed-formula} is not satisfiable. By
Lemma~\ref{lem:support-obstruction}, the graph $Q$ is acyclic and every
negative target is reachable and strictly dominates the index of its
clause. Strictness follows from looplessness.

Let $\nu$ be the number of reachable variables, let $E_+$ and $E_-$
count positive and negative support occurrences in the reachable tail
clauses, and let $J$ be the number of reachable variables of indegree two.
A positive support entering $j$ records an arc leaving $h_j$. Legality
shows that distinct ordinary support arcs entering $j$ have distinct
tails. The source tail $s$ is different from every variable.

The source arc $s\to k$ is the only source arc, so every directed path
from $s$ begins with this arc. If an ordinary positive support entered
$k$, its tail would be reachable from $k$ and would close a directed cycle.
Thus $k$ has indegree one, while every other reachable variable has
indegree at most two. Consequently,
\begin{equation}
 E_++1=\nu+J.
 \label{eq:prescribed-indegree}
\end{equation}
Every reachable tail clause has exactly two support positions. Positive
targets are reachable by definition, and negative targets are reachable
by Lemma~\ref{lem:support-obstruction}. Hence
\[
 E_++E_-=2\nu.
\]
Together with~\eqref{eq:prescribed-indegree}, this gives
\begin{equation}
 E_-=\nu+1-J.
 \label{eq:prescribed-negative}
\end{equation}

Each negative target occurs at most once, because the matched arc
$t_j\to h_j$ already uses one of the two arcs leaving $t_j$. Thus the
$E_-$ negative targets are distinct nonleaves of the dominator tree. By
Lemma~\ref{lem:dominator-leaves}, this tree has at least $J+1$ variable
leaves and at most $\nu-J-1$ nonleaf variables, contradicting
\eqref{eq:prescribed-negative}. The formula is therefore satisfiable.

Switch exactly the true variables. If no lift vertex at an index $t_i$
is selected, then $x_{t_i}$ is dominated by $y_{h_i}$ and $y_{t_i}$ is
dominated by the support in its clause. At every remaining index, at least
one lift vertex is selected. Whenever exactly one lift vertex at an index
is selected, the other is dominated through the fixed edge. The resulting
edges form a dominating matching and the switch at $k$ is made. The second
statement follows by reversing all arcs and interchanging
the two partite classes.
\end{proof}

\section{Boundary states at two-edge cuts}\label{sec:boundary-transfer}

\subsection{A boundary state with both ends selected}\label{sec:edge-II}

\begin{definition}\label{def:II-boundary}
Let $F=(X,Y;E)$ be a simple three-edge-connected cubic bipartite graph,
let $e=ab\in E(F)$ with $a\in X$ and $b\in Y$, and put $C=F-e$. An
\emph{$\sII$ boundary realization} of $C$ is a paired dominating set of
$C$ that contains $a$ and $b$, together with a perfect matching of its
induced subgraph that covers both vertices.
\end{definition}

\begin{proposition}\label{prop:edge-II}
Let $F=(X,Y;E)$ be a simple three-edge-connected cubic bipartite graph, let
$e=ab\in E(F)$ with $a\in X$ and $b\in Y$, and put $C=F-e$. If
$|V(F)|$ is divisible by four, then $C$ has an $\sII$ boundary realization
of size $|V(F)|/2$.
\end{proposition}

\begin{proof}
Decompose $F$ into three perfect matchings, and choose one, say $M$, that
avoids $e$. Contract $M$ and let $\vec H$ be the resulting legal
digraph. Write the arc corresponding to $e$ as $A\to B$. Since $a=x_A$, let
$f=A\to C_0$ be the other out-arc of $A$ corresponding to an edge incident
with $a$ outside $M$.

The underlying multigraph $H$ has even order and is $4$-edge-connected.
By Lemma~\ref{lem:uniform-even}, it has a perfect matching $P$ containing
$f$. Since $f$ meets $A\to B$, the matching $P$ avoids $A\to B$.
Orient the edges of $P$ as $t_i\to h_i$. Then $A$ is a tail of $P$.

If $B$ is a head of $P$, use the standard lift of $P$. It selects
$a=x_A$ and $b=y_B$, and both are covered by the lifted matching. Suppose
instead that $B=t_k$ is a tail. Apply the tail version of
Theorem~\ref{thm:prescribed-switch} at index $k$. Every tail remains
selected on the $X$-side, so $a$ is selected. Since the switch at $k$ is
prescribed, the fixed edge $x_B y_B$ is used, and hence $b$ is also selected
and covered.

In both cases the matching avoids $e$ and dominates $F$. Deleting $e$
does not affect domination because both ends of $e$ are selected. If
$|V(F)|=4k$, then $|V(H)|=2k$ and the matching has $k$ edges. Its endpoint
set therefore has size $2k=|V(F)|/2$.
\end{proof}

\subsection{Four boundary symbols}\label{sec:two-cut}

\begin{lemma}
\label{lem:boundary-parts}
Let $C$ be a side of a two-edge cut in a cubic bipartite graph with
partite sets $X$ and $Y$. The two ends of the cut edges that lie in $C$
are distinct and belong to different partite classes.
\end{lemma}

\begin{proof}
Let $d_X$ and $d_Y$ be the numbers of cut edges whose end in $C$ lies in
$X$ and $Y$, respectively. Then $d_X+d_Y=2$. Counting the ends of
internal edges of $C$ in the two partite classes gives
$3(|X\cap V(C)|-|Y\cap V(C)|)=d_X-d_Y$. The right side is one of $-2,0,2$. It is divisible by $3$ only when it is
zero. Hence $d_X=d_Y=1$. Thus the two cut edges meet $C$ at two
distinct vertices, one in each partite class.
\end{proof}

\begin{definition}
\label{def:boundary-symbols}
A \emph{two-terminal side} is a bipartite graph
$C=(X_C,Y_C;E_C)$ with two distinguished boundary vertices
$a\in X_C$ and $b\in Y_C$. The vertices $a$ and $b$ have degree two,
and every other vertex has degree three. By
Lemma~\ref{lem:boundary-parts}, each side of a two-edge cut in a cubic
bipartite graph is naturally a two-terminal side, with the ends of the cut
as its boundary vertices.

Put $\Omega=\{\OO,\II,\DD,\RR\}$. For a boundary vertex $v$, the
four symbols have the following meanings:
\begin{enumerate}[label=\textnormal{(\alph*)}]
\item $\OO$: $v$ is selected but is not covered by the local matching, so
      it must be matched through the cut;
\item $\II$: $v$ is selected and covered by the local matching;
\item $\DD$: $v$ is not selected and has a selected neighbor in $C$;
\item $\RR$: $v$ is not selected and has no selected neighbor in $C$, so
      its neighbor across the cut must be selected.
\end{enumerate}
A \emph{boundary state} of $C$ is an ordered pair
$\sigma=(\sigma_a,\sigma_b)\in\Omega^2$.
\end{definition}

\begin{definition}\label{def:local-realization}
Let $C$ be a two-terminal side with boundary vertices $a$ and $b$. For
$\sigma=(\sigma_a,\sigma_b)\in\Omega^2$, let
$O_\sigma=\{w\in\{a,b\}:\sigma_w=\OO\}$. A \emph{local realization} of $\sigma$ is a set $S\subseteq V(C)$ together
with a matching $K\subseteq E(C[S])$ such that
$V(K)=S\setminus O_\sigma$ and the following boundary and domination
conditions hold:
\begin{enumerate}[label=\textnormal{(B\arabic*)},leftmargin=3.2em]
\item $\sigma_w=\II$ exactly when $w\in V(K)$, and
      $\sigma_w=\OO$ exactly when $w\in S\setminus V(K)$;
\item $\sigma_w=\DD$ exactly when $w\notin S$ and
      $N_C(w)\cap S\ne\varnothing$, while $\sigma_w=\RR$ exactly when
      $w\notin S$ and $N_C(w)\cap S=\varnothing$;
\item every vertex of $V(C)\setminus(S\cup\{a,b\})$ has a neighbor in
      $S$.
\end{enumerate}
Thus every endpoint of a local matching edge is selected, every selected
vertex other than an $\OO$ boundary vertex is covered exactly once, and no
unselected vertex is used by the local matching. Define
$c_C(\sigma)=\min |S|$, and put $c_C(\sigma)=+\infty$ if no local
realization exists.
\end{definition}

\begin{definition}\label{def:compatibility}
Two symbols on the ends of a cut edge are \emph{compatible}, written
$\sim$. The compatible pairs are $\OO\sim\OO$,
$\II\sim\II,\DD,\RR$, $\DD\sim\II,\DD$, and $\RR\sim\II$. The relation
is symmetric.
\end{definition}

\begin{definition}\label{def:crossed-coordinates}
For every two-terminal side, the first coordinate belongs to its boundary
vertex in $X$, and the second coordinate belongs to its boundary vertex in
$Y$. For sides $C$ and $D$ on the two sides of a cut, write the cut edges
as $a_C b_D$ and $b_C a_D$, where $a_C,a_D\in X$ and $b_C,b_D\in Y$. Thus
states $\sigma=(\sigma_a,\sigma_b)$ on $C$ and
$\tau=(\tau_a,\tau_b)$ on $D$ are compatible exactly when
$\sigma_a\sim\tau_b$ and $\sigma_b\sim\tau_a$.
\end{definition}

The following gluing identity follows directly from the boundary-state
definitions. We give both directions because the identity is used throughout
the remainder of the proof.

\begin{lemma}[Gluing identity]\label{lem:two-cut-identity}
Let $G=(X,Y;E)$ be a finite cubic bipartite graph. Suppose that a two-edge
cut separates $G$ into sides $C_1$ and $C_2$, with cut edges $a_1 b_2$ and
$b_1 a_2$, where $a_i\in X\cap V(C_i)$ and
$b_i\in Y\cap V(C_i)$. Then
\begin{equation}
 \gpr(G)=
 \min_{\alpha_1\sim\beta_2,\,\beta_1\sim\alpha_2}
 \bigl(c_{C_1}(\alpha_1,\beta_1)
      +c_{C_2}(\alpha_2,\beta_2)\bigr).
 \label{eq:two-cut-minplus}
\end{equation}
\end{lemma}

\begin{proof}
First choose compatible local realizations $(S_i,K_i)$ on the two sides.
Take $K_1\cup K_2$, and add a cut edge when both ends of that edge have
symbol $\OO$. The resulting matching covers every selected vertex. By
Definition~\ref{def:compatibility}, an $\RR$ symbol faces an $\II$ symbol,
so every boundary request is satisfied. All other domination conditions
hold inside the sides. Hence $S_1\cup S_2$ is paired dominating. This
gives the upper bound in~\eqref{eq:two-cut-minplus}.

For the reverse bound, let $S$ be a paired dominating set of $G$, and let
$K$ be a perfect matching of $G[S]$. Restrict $S$ and the internal edges
of $K$ to each side. A boundary vertex matched through the cut receives
$\OO$. A selected boundary vertex matched inside its side receives
$\II$. An unselected boundary vertex receives $\DD$ if it has a selected
neighbor inside its side; otherwise it receives $\RR$. In the latter
case, global domination forces its neighbor across the cut to be selected.

These two local states are compatible. Indeed, if a cut edge belongs to
$K$, both ends receive $\OO$. Otherwise, an $\RR$ end must face a selected
end. That selected end is not open and therefore has symbol $\II$. The
two local costs add to $|S|$, which proves the reverse bound.
\end{proof}

\begin{definition}\label{def:side-closure}
Let $C$ be a two-terminal side with boundary vertices $a$ and $b$. Add a
new marked edge instance $e_C^\star$ with ends $a,b$, and write
$C^+=C+e_C^\star$. The graph $C^+$ is the \emph{closure} of $C$. If $ab\in E(C)$, then
$e_C^\star$ is parallel to the old edge $ab$.
\end{definition}

The next formula is a direct consequence of the definitions, but recording it
will simplify the later arguments.

\begin{proposition}[Closure formula]\label{prop:closure-formula}
For every two-terminal side $C$,
\begin{equation}
 \begin{split}
 \gpr(C^+)=\min\{&c_C(\sOO),c_C(\sII),c_C(\sID),
 c_C(\sIR),\\
 &c_C(\sDI),c_C(\sDD),c_C(\sRI)\}.
 \end{split}
 \label{eq:closure-formula}
\end{equation}
\end{proposition}

\begin{proof}
Let $S$ be paired dominating in $C^+$, and let $K$ be a perfect matching of
$C^+[S]$. If $e_C^\star\in K$, deleting this edge gives the state
$\sOO$. Suppose $e_C^\star\notin K$. Then neither boundary vertex is
open. A request at one boundary vertex can only be met through
$e_C^\star$, so the other boundary vertex must have symbol $\II$. The six
possible states are the six states with no open boundary vertex listed in
\eqref{eq:closure-formula}.

Conversely, an $\sOO$ realization extends to $C^+$ by adding
$e_C^\star$ to its matching. Each of the other six realizations extends
without using $e_C^\star$; any request is met by the selected opposite
boundary vertex.
Thus all seven states give paired dominating sets of $C^+$, and no other
state can occur.
\end{proof}

\subsection{Minimal sides and state transfer}\label{sec:profiles}

\begin{definition}\label{def:state-families}
In this section, a two-terminal side $C=(X_C,Y_C;E_C)$ has
boundary vertices $a\in X_C$ and $b\in Y_C$. Put
$\ClStates=\{\sOO,\sII,\sID,\sIR,\sDI,\sDD,\sRI\}$ and
$\NonOpenStates=\ClStates\setminus\{\sOO\}$. The states in $\ClStates$
are exactly those whose local realizations extend to paired dominating sets
of $C^+$. The states in $\NonOpenStates$ have no open boundary vertex. We
also put $\mathcal U=\{\sOO\}\cup
(\{\II,\DD,\RR\}\times\{\II,\DD,\RR\})$,
$\mathcal S=\{\sDD,\sRD,\sDR,\sRR\}$, and
$\mathcal G=\{\sII,\sID,\sDI,\sDD\}$.
\end{definition}

\begin{definition}\label{def:basic-side}
A two-terminal side $C$ is \emph{basic} if its closure $C^+$ is simple
and three-edge-connected.
\end{definition}

\begin{lemma}\label{lem:minimum-side-basic}
Let $G$ be a simple connected cubic bipartite graph with a two-edge cut.
Choose a side $C$ of minimum order among all sides of all two-edge cuts.
Then $C^+$ is simple and three-edge-connected.
\end{lemma}

\begin{proof}
First, $C$ is connected. Otherwise its components would have total edge
boundary two in $G$, so one component would have boundary at most one.
Boundary zero contradicts the connectedness of $G$, and boundary one
contradicts Lemma~\ref{lem:cubic-bipartite-bridgeless}. Hence $C^+$ is a
connected cubic bipartite multigraph with no bridge.

Suppose that $C^+$ has a two-edge cut $\delta_{C^+}(U)$, where
$\varnothing\ne U\subsetneq V(C)$. If $e_C^\star$ does not belong to
this cut, then $a$ and $b$ lie on the same side. Choose $U$ to be the other
side. No original cut edge of $G$ is incident with a vertex of $U$, and
therefore $\delta_G(U)=\delta_{C^+}(U)$. Thus $G[U]$ is a smaller side of a two-edge cut of $G$.

Suppose instead that $e_C^\star\in\delta_{C^+}(U)$. Then $U$ contains
exactly one of $a,b$; call it $w$. Let $f_w$ be the original cut edge of
$G$ incident with $w$. We then have
$\delta_G(U)=(\delta_{C^+}(U)\setminus\{e_C^\star\})\cup\{f_w\}$.
Again $G[U]$ is a side of a two-edge cut and has fewer vertices than $C$.
Both cases contradict the choice of $C$. Thus $C^+$ is
three-edge-connected.

The only possible parallel pair in $C^+$ consists of $e_C^\star$ and an
old edge $ab$. If this pair exists, the two other edges incident with
$\{a,b\}$ form a two-edge cut of $C^+$. This is impossible. Therefore
$C^+$ is simple, and $C$ is basic.
\end{proof}

\begin{lemma}\label{lem:even-basic-profile}
Let $F=(X,Y;E)$ be a simple three-edge-connected cubic bipartite graph with
$|V(F)|\equiv0\pmod4$, and let $e=ab\in E(F)$, where $a\in X$ and
$b\in Y$. For $C=F-e$, we have
$c_C(\sII),c_C(\sOO)\le |V(F)|/2$. Thus, for every $\tau\in\mathcal U$, the side $C$ has a state
compatible with $\tau$ at cost at most $|V(F)|/2$.
\end{lemma}

\begin{proof}
The bound for $\sII$ is Proposition~\ref{prop:edge-II}. To obtain $\sOO$,
choose a perfect matching $M$ of $F$ containing $e$, and let $\vec H$ be the
legal contraction of $(F,M)$. Let $H$ be its underlying multigraph and let
$z$ be the vertex corresponding to $e$. By Lemma~\ref{lem:cut}, the
even-order multigraph $H$ is $4$-edge-connected. Hence
Lemma~\ref{lem:uniform-even} gives a perfect matching $P$ of $H$.

Let the edge of $P$ covering $z$ be oriented as it occurs in the legal
digraph. If $z$ is its head, use the prescribed head switch. If $z$ is
its tail, use the prescribed tail switch. The resulting dominating
matching contains the fixed edge $e$. Remove $e$ from the local matching
but keep its two ends selected. This gives an $\sOO$ realization of cost
$|V(F)|/2$.

The state $\sOO$ is compatible with $\sOO$, and $\sII$ is compatible with
every state having no open coordinate. Thus, for a given
$\tau\in\mathcal U$, use the $\sOO$ realization when $\tau=\sOO$ and
the $\sII$ realization otherwise.
\end{proof}

\begin{lemma}\label{lem:odd-basic-profile}
Let $F=(X,Y;E)$ be a simple three-edge-connected cubic bipartite graph with
$|V(F)|=4k+2$, and let $e=ab\in E(F)$, where $a\in X$ and $b\in Y$.
For $C=F-e$, some
$\sigma\in\NonOpenStates$ satisfies
$c_C(\sigma)\le 2k=|V(F)|/2-1$.
\end{lemma}

\begin{proof}
Choose a perfect matching $M$ of $F$ that avoids $e$, and let $\vec H$ be
the legal contraction of $(F,M)$. The edge $e$ becomes an arc $A\to B$.
By Lemma~\ref{lem:cut}, the underlying multigraph of $\vec H$ is
$4$-edge-connected. Lemma~\ref{lem:directional-certificate} then gives a
near-perfect matching that avoids $A\to B$ and has a directional
certificate. Theorem~\ref{thm:one-sided-switch} gives a dominating
matching of $F$ with $k$ edges that uses only fixed edges from $M$ and
lifts of edges in the near-perfect matching. It therefore avoids $e$. After $e$ is
deleted, every internal vertex remains dominated. A selected boundary
vertex has symbol $\II$, and an unselected boundary vertex has symbol
$\DD$ or $\RR$. If one end has symbol $\RR$, then it was dominated in
$F$ only through $e$, so the other end is selected and has symbol $\II$.
Thus the resulting state belongs to $\NonOpenStates$ and has cost $2k$.
\end{proof}

\begin{definition}\label{def:bond-extension}
Let $D=(X_D,Y_D;E_D)$ be a two-terminal side with boundary vertices
$x\in X_D$ and $y\in Y_D$. Add a vertex $u$ to the first partite class
and a vertex $v$ to the second, together with the edges $uv$, $uy$, and
$vx$. The resulting side $R$, whose boundary vertices are $u$ and $v$,
is the \emph{two-vertex extension} of $D$.

Assume that $\tau\in\Omega^2$ has a local realization in $D$. For
$\rho\in\Omega^2$, write $\tau\xrightarrow{d}\rho$ if every local
realization of $\tau$ in $D$ extends to a local realization of $\rho$ in
$R$, without changing its selected vertices or matching edges in $D$, after
exactly $d$ vertices from $\{u,v\}$ are selected.
\end{definition}

\begin{theorem}
\label{thm:exact-state-transfer}
Let $R$ be the two-vertex extension of a two-terminal side $D$. For every state
$\tau\in\Omega^2$ that has a local realization in $D$, the complete list
of transfers from $\tau$ is given in Table~\ref{tab:full-state-transfer}.
In the row indexed by $\tau$, the entry $\rho[d]$ denotes
$\tau\xrightarrow{d}\rho$. No other outer state is possible while the
inner realization is left unchanged, and each value of $d$ in the table is
smallest possible.
\end{theorem}

\begingroup
\small
\setlength{\tabcolsep}{5pt}
\renewcommand{\arraystretch}{1.08}
\begin{longtable}{@{}c >{\raggedright\arraybackslash}p{0.76\textwidth}@{}}
\caption{Exact transfers through one two-vertex extension.}
\label{tab:full-state-transfer}\\
\toprule
Inner state & Outer states and added costs\\
\midrule
\endfirsthead
\toprule
Inner state & Outer states and added costs\\
\midrule
\endhead
\midrule
\multicolumn{2}{r}{Continued on the next page}\\
\endfoot
\bottomrule
\endlastfoot
$\sOO$ & $\sII[2]$\\
$\sOI$ & $\sDI[1],\ \sOI[2]$\\
$\sOD$ & $\sDI[1],\ \sOI[2]$\\
$\sOR$ & $\sOI[2]$\\
$\sIO$ & $\sID[1],\ \sIO[2]$\\
$\sII$ & $\sDD[0],\ \sDO[1],\ \sOD[1],\ \sOO[2],\ \sII[2]$\\
$\sID$ & $\sRD[0],\ \sDO[1],\ \sOD[1],\ \sOO[2],\ \sII[2]$\\
$\sIR$ & $\sOD[1],\ \sOO[2],\ \sII[2]$\\
$\sDO$ & $\sID[1],\ \sIO[2]$\\
$\sDI$ & $\sDR[0],\ \sDO[1],\ \sOD[1],\ \sOO[2],\ \sII[2]$\\
$\sDD$ & $\sRR[0],\ \sDO[1],\ \sOD[1],\ \sOO[2],\ \sII[2]$\\
$\sDR$ & $\sOD[1],\ \sOO[2],\ \sII[2]$\\
$\sRO$ & $\sIO[2]$\\
$\sRI$ & $\sDO[1],\ \sOO[2],\ \sII[2]$\\
$\sRD$ & $\sDO[1],\ \sOO[2],\ \sII[2]$\\
$\sRR$ & $\sOO[2],\ \sII[2]$\\
\end{longtable}
\endgroup

\begin{proof}
Fix a local realization $(S,K)$ of an inner state $\tau$ in $D$. Let
$Z\subseteq\{u,v\}$ be the new selected vertices, and let
$L\subseteq\{uv,uy,vx\}$ be the new matching edges. Put
$S'=S\cup Z$ and $K'=K\cup L$. The pair $(S',K')$ is a local realization
in $R$ exactly when the following conditions hold:
\begin{enumerate}[label=\textnormal{(T\arabic*)},leftmargin=3.1em]
\item $K'$ is a matching in $R[S']$;
\item an $\OO$ symbol at $x$ forces $vx\in L$, and an $\OO$ symbol at
      $y$ forces $uy\in L$;
\item an $\RR$ symbol at $x$ forces $v\in Z$, and an $\RR$ symbol at $y$
      forces $u\in Z$;
\item a selected outer vertex has symbol $\II$ or $\OO$ according as it is
      covered or not covered by $L$, while an unselected outer vertex has
      symbol $\DD$ or $\RR$ according as it has or has no selected
      neighbor in $R$.
\end{enumerate}
These conditions give the exact matching requirement
$V(K')=S'\setminus O_\rho$, and they also give all boundary domination
requirements. They depend only on the inner state $\tau$ and on the pair
$(Z,L)$, not on the chosen local realization $(S,K)$. Thus every legal
pair $(Z,L)$ extends every local realization of $\tau$.

There are four choices for $Z$ and eight choices for $L$. For each
inner state, Appendix~\ref{app:bond-check} checks all $32$ pairs
$(Z,L)$ against conditions \textnormal{(T1)}--\textnormal{(T4)}. Once a
legal pair is fixed, the outer state and the added cost $d=|Z|$ are fixed.
The choice of $(S,K)$ was arbitrary, so the check proves every transfer in
Table~\ref{tab:full-state-transfer} and excludes every other transfer. In
each row, the absence of a legal pair with smaller $|Z|$ also proves the
claimed minimum cost. Thus the table is a finite consequence of the four
local rules above and does not depend on a particular realization of the
inner state.
\end{proof}

\subsection{State transfer along an extension chain}\label{sec:chain-transfer}

\begin{definition}\label{def:bond-chain}
A \emph{extension chain} is a sequence
$D_0,D_1,\ldots,D_r$ in which $D_i$ is the two-vertex extension of $D_{i-1}$
for $1\le i\le r$.
\end{definition}

The next lemma lists the parts of Table~\ref{tab:full-state-transfer}
that will be used later. Compatibility is taken across the crossed cut edges
in Lemma~\ref{lem:two-cut-identity}.

\begin{lemma}\label{lem:chain-bounds}
Let $D_0,D_1,\ldots,D_r$ be an extension chain, and suppose that
$\tau\in\ClStates$ has a local realization in $D_0$ of cost $q$.
\begin{enumerate}[label=\textnormal{(\roman*)},leftmargin=2.8em]
\item If $r$ is even, $D_r$ has a state in $\mathcal U$ of cost at most
      $q+r$. If $r$ is odd, it has a state in $\mathcal U$ of cost at most
      $q+r+1$.
\item For every $\sigma\in\NonOpenStates$ and odd $r$, $D_r$ has a state
      compatible with $\sigma$ of cost at most $q+r+1$.
\item For every $\sigma\in\NonOpenStates$ and even $r\ge2$, $D_r$ has a state
      compatible with $\sigma$ of cost at most $q+r+2$.
\item If $\tau=\sII$, then for even $r$ the state $\sII$ is realizable
      in $D_r$ with cost at most $q+r$, while for odd $r$ a state in
      $\mathcal U$ is realizable with cost at most $q+r-1$.
\item If $\tau\in\mathcal G$, then for every odd $r\ge1$, $D_r$ has a
      state in $\mathcal S$ of cost at most $q+r-1$; for every even
      $r\ge2$, $D_r$ has an $\sII$ state of cost at most $q+r$.
\end{enumerate}
\end{lemma}

\begin{proof}
For $r=0$, part \textnormal{(i)} follows from
$\ClStates\subseteq\mathcal U$, and the even case of \textnormal{(iv)} is
immediate. We may therefore assume that every path used below has positive
length.

We read the arrows from Table~\ref{tab:full-state-transfer}. Every state in
$\ClStates$ reaches $\sII$ in one step at cost two. In two steps, the
following paths end in $\mathcal U$ and have cost at most two:
\[
\begin{array}{lll}
\sOO\to\sII\to\sDD,&
\sII\to\sDD\to\sRR,&
\sID\to\sRD\to\sII,\\
\sIR\to\sOD\to\sDI,&
\sDI\to\sDR\to\sII,&
\sDD\to\sRR\to\sOO,\\
\sRI\to\sDO\to\sID.&
\end{array}
\]
In three steps, the path
$\tau\to\sII\to\sDD\to\sRR$ has cost two for every initial state
$\tau\in\ClStates$. From any state in $\mathcal U$, a four-step path to
$\sOO$ has cost four. For a state with no open boundary vertex, first move to $\sII$ at cost
two, and then use $\sII\to\sDD\to\sRR\to\sOO$ at cost two. From $\sOO$, use
$\sOO\to\sII\to\sDD\to\sRR\to\sOO$. These paths, followed by any number
of four-step repetitions, prove \textnormal{(i)}.

For \textnormal{(ii)}, use $\tau\to\sII$ at length one and
$\tau\to\sII\to\sDD\to\sII$ at length three. The state $\sII$ is
compatible with every state in $\NonOpenStates$. Repeating the four-step
cycle $\sII\to\sDD\to\sRR\to\sOO\to\sII$ proves the statement for all odd lengths.

For \textnormal{(iii)}, use $\tau\to\sII\to\sII$ at length two. At length
four, use $\tau\to\sII\to\sDD\to\sRR\to\sII$. Both paths have cost four.
The same four-step cycle gives every larger even length.

For \textnormal{(iv)}, repeat the two-step path
$\sII\to\sDD\to\sII$, which has cost two. This gives the even case. One
additional zero-cost arrow $\sII\to\sDD$ gives the odd case.

For \textnormal{(v)}, the zero-cost arrows are
$\sII\to\sDD$, $\sID\to\sRD$, $\sDI\to\sDR$, and
$\sDD\to\sRR$. Thus every state in $\mathcal G$ reaches $\mathcal S$ in one step at no
cost. Every state in $\mathcal S$ then reaches $\sII$ in one step at cost
two, and $\sII\to\sDD$ has cost zero. Repeating this two-step pattern
proves the odd case. For a positive even length, stop at the intermediate
$\sII$ state.
\end{proof}

\section{The Gallai--Edmonds decomposition and the main proof}\label{sec:global-proof}

\subsection{Switching in a factor-critical subgraph}\label{sec:gallai}

We now prove a switching lemma for a factor-critical subgraph. All
neighbors outside this subgraph have already been selected.

\begin{lemma}\label{lem:relative-switch}
Let $\vec L$ be a legal balanced digraph, and let
$\vec Q$ be obtained from $\vec L$ by deleting at most one arc. Let $Q$
be the underlying multigraph of $\vec Q$. Suppose that
$W\subseteq V(Q)$ and $Q[W]$ is factor-critical. Let
$A\subseteq V(Q)\setminus W$ contain every neighbor of $W$ outside $W$.
In the canonical lift $L(\vec L,\vec Q)$, assume that the fixed edge
$x_a y_a$ has already been selected for every $a\in A$.

If $\vec Q$ contains an arc $z\to a_0$ with $z\in W$ and $a_0\in A$,
then there is a matching $K_W$ of $(|W|-1)/2$ edges, using only lift
vertices indexed by $W$, such that
$K_W\cup\{x_a y_a:a\in A\}$ is a matching and its endpoint set dominates
every lift vertex indexed by $W$. The same conclusion holds when
$\vec Q$ contains the arc $a_0\to z$.
\end{lemma}

\begin{proof}
We prove the case $z\to a_0$. The other case follows by reversing all
arcs and interchanging the two partite classes.

If $W=\{z\}$, take $K_W=\varnothing$. The selected vertex $y_{a_0}$
dominates $x_z$. The vertex $z$ has two distinct in-neighbors in $\vec L$,
and at most one arc was deleted. Hence $\vec Q$ contains an arc $a\to z$
for some $a\notin W$. By the choice of $A$, we have $a\in A$, and the
selected vertex $x_a$ dominates $y_z$. We may therefore assume
$|W|\ge3$.

\medskip
\noindent\emph{Step 1: the relative switching formula.}
Since $Q[W]$ is factor-critical, $Q[W]-z$ has a perfect matching. Write
its edge instances in their directions as
$P=\{t_i\to h_i:1\le i\le r\}$, where $r=(|W|-1)/2$. Use the
head-switching variables from Definition~\ref{def:switching-support}. The
selected $Y$-indices are the vertices $h_i$, and the selected $X$-indices
are $\{t_i:\xi_i=0\}\cup\{h_i:\xi_i=1\}$. Since $z\to a_0$ is present
in $\vec Q$, the selected vertex $y_{a_0}$ dominates $x_z$. Thus only
$y_z$ still needs to be dominated at the exposed index. If $\xi_i=1$,
then $x_{t_i}$ is dominated by $y_{h_i}$, while $y_{t_i}$ needs a
selected $X$-neighbor.

Extend the support map to $A\cup W$ by
\[
 \widehat\lambda(u)=
 \begin{cases}
  1,&u\in A,\\
  0,&u=z,\\
  \lambda(u),&u\in W\setminus\{z\}.
 \end{cases}
\]
Every neighbor of $W$ lies in $A\cup W$. In $\vec Q$, define
\[
 C_z=\bigvee_{u\in\Nin_{\vec Q}(z)}\widehat\lambda(u),
 \qquad
 C_i=\neg\xi_i\vee
      \bigvee_{u\in\Nin_{\vec Q}(t_i)}\widehat\lambda(u),
\]
and put $\Phi_W=C_z\wedge\bigwedge_{i=1}^{r}C_i$. These are exactly the
remaining domination conditions. An in-arc from $A$ contributes the
constant $1$, a present in-arc from $z$ contributes the constant $0$, and
the deleted arc contributes no support position.

If $C_z$ contains a negative support or the constant $1$, the all-zero
assignment satisfies $\Phi_W$. We may therefore assume that $C_z$
consists of $s_0\in\{1,2\}$ positive supports. Here $s_0\ge1$, because
at most one arc was deleted, and no zero support occurs in $C_z$ because
$\vec L$ is loopless.

\medskip
\noindent\emph{Step 2: the support graph.}
Add a source $s$ and one source arc for each positive support in $C_z$.
A positive support $\xi_j$ in $C_i$ gives the arc $i\to j$, and a
negative support $\neg\xi_j$ is denoted by $i\to\overline j$. Let $R$
be the subgraph induced by $s$ and the variables reachable from $s$ by
positive arcs. Suppose that $\Phi_W$ is not satisfiable. By
Lemma~\ref{lem:support-obstruction}, the graph $R$ is acyclic, no
reachable tail clause contains the constant $1$, and every negative
target is reachable and strictly dominates the index of its clause.
Strictness follows from looplessness.

Let $\nu$ be the number of reachable variables. Let $E_+$ and $E_-$ be
the numbers of positive and negative support occurrences in the reachable
tail clauses, and let $J$ be the number of reachable variables of
indegree two, with source arcs included. A support arc entering $j$ comes
from an arc of $\vec Q$ leaving $h_j$. The two out-neighbors of $h_j$ are
distinct, so ordinary support arcs entering $j$ have distinct tails. The
positive supports in $C_z$ involve distinct variables because the
in-neighbors of $z$ and the matched heads are distinct. Thus at most one
source arc enters a variable. If a source arc enters $j$, the arc
$h_j\to z$ uses one out-arc of $h_j$, so at most one ordinary support arc
also enters $j$. Hence every reachable variable has indegree one or two,
and distinct entering arcs have distinct tails. Counting entering arcs
gives
\begin{equation}
 E_++s_0=\nu+J.
 \label{eq:relative-indegree}
\end{equation}

\medskip
\noindent\emph{Step 3: deficient support positions.}
Each reachable tail clause has two original support positions in $\vec L$.
Call a position \emph{deficient} if its arc is deleted, or if its arc is
present in $\vec Q$ and has tail $z$, in which case it contributes the
constant $0$. Let $\Delta$ be the number of deficient positions among the
reachable tail clauses. We claim that
\begin{equation}
 \Delta\le s_0.
 \label{eq:relative-deficiency}
\end{equation}

Suppose first that $s_0=1$. The other original in-arc of $z$ cannot give
a positive support, a negative support, the constant $1$, or the constant
$0$. These four possibilities would, respectively, give $s_0=2$, make the
all-zero assignment work, make the all-zero assignment work, or create a
loop. Hence that in-arc is the deleted arc, so the deletion does not remove
a support from a tail clause. Moreover, $z\to a_0$ uses one of the two
out-arcs of $z$. At most one support in a reachable tail clause is therefore
the constant $0$, and $\Delta\le1=s_0$.

Suppose next that $s_0=2$. No deleted arc enters $z$. At most one out-arc
of $z$ contributes the constant $0$, because the other is $z\to a_0$,
and the deleted arc removes at most one further support position. Hence
$\Delta\le2=s_0$. This proves \eqref{eq:relative-deficiency}.

No reachable tail clause contains the constant $1$. Every nondeficient
support position is therefore a positive or negative literal. Its target
is reachable directly in the positive case and by
Lemma~\ref{lem:support-obstruction}\textnormal{(iii)} in the negative
case. Consequently,
\begin{equation}
 E_++E_-=2\nu-\Delta.
 \label{eq:relative-support-count}
\end{equation}
Together with \eqref{eq:relative-indegree},
\eqref{eq:relative-deficiency}, and \eqref{eq:relative-support-count},
this gives
\begin{equation}
 E_-=\nu+s_0-J-\Delta\ge\nu-J.
 \label{eq:relative-negative-count}
\end{equation}

\medskip
\noindent\emph{Step 4: the dominator-tree contradiction.}
A negative target $j$ comes from an arc leaving $t_j$. The matched arc
$t_j\to h_j$ already uses one out-arc of $t_j$, so each target occurs at
most once. Every negative target is a nonleaf of the dominator tree. By
Lemma~\ref{lem:dominator-leaves}, the tree has at least $J+1$ variable
leaves and at most $\nu-J-1$ nonleaf variables. This contradicts
\eqref{eq:relative-negative-count}. Hence $\Phi_W$ is satisfiable.

For a satisfying assignment, let $K_W$ contain $x_{t_i} y_{h_i}$ when
$\xi_i=0$ and $x_{h_i} y_{h_i}$ when $\xi_i=1$. These edges are pairwise
disjoint and use only indices in $W$. Since $A\cap W=\varnothing$, their
union with the selected fixed edges at $A$ is also a matching. The clauses
dominate both lift vertices at every omitted tail index. At the exposed
index, $x_z$ is dominated by $y_{a_0}$ and $y_z$ is dominated by $C_z$.
At every remaining index in $W$, at least one lift vertex is selected.
Whenever exactly one lift vertex at an index is selected, the other is
dominated through the fixed edge. This proves the lemma.
\end{proof}

\subsection{The Gallai--Edmonds construction}

\begin{definition}\label{def:gallai-edmonds}
Let $Q$ be a graph. Let $D_{\rm GE}$ be the set of vertices missed by at
least one maximum matching of $Q$. Put
$A_{\rm GE}=N_Q(D_{\rm GE})\setminus D_{\rm GE}$ and
$C_{\rm GE}=V(Q)\setminus(D_{\rm GE}\cup A_{\rm GE})$. The three sets form the \emph{Gallai--Edmonds decomposition} of $Q$.
For a multigraph, we apply this definition to its underlying simple graph.
\end{definition}

We use the Gallai--Edmonds Structure Theorem as stated by
Lov\'asz and Plummer~\cite[Theorem~3.2.1]{LovaszPlummer1986}. Every
component of $Q[D_{\rm GE}]$ is factor-critical,
$Q[C_{\rm GE}]$ has a perfect matching, and there is no edge from
$D_{\rm GE}$ to $C_{\rm GE}$. If $Q[D_{\rm GE}]$ has $s$ components,
then every maximum matching leaves $\delta=s-|A_{\rm GE}|$ vertices
exposed. For a multigraph, replacing each parallel class by one edge
preserves matching size and the possible sets of covered vertices. It
therefore preserves the Gallai--Edmonds sets, components,
perfect-matchability, and factor-criticality. When the theorem supplies a
matching edge, we choose a corresponding edge instance of the original
multigraph and keep its direction in the lift.

\begin{theorem}\label{thm:gallai-lift}
Let $\vec L$ be a legal balanced digraph of even order
$m$, and let $\vec Q$ be obtained from $\vec L$ by deleting at most one
arc. Let $Q$ be the underlying multigraph of $\vec Q$, and suppose that
$Q$ is connected. If $Q$ has no perfect matching, then the canonical lift
$L(\vec L,\vec Q)$ has a dominating matching whose endpoint set has size
at most $m-2$.
\end{theorem}

\begin{proof}
Apply the Gallai--Edmonds Structure Theorem
\cite[Theorem~3.2.1]{LovaszPlummer1986} to $Q$, and let
$W_1,\ldots,W_s$ be the components of $Q[D_{\rm GE}]$.
We first note that $A_{\rm GE}\ne\varnothing$. Otherwise, since there is no
edge from $D_{\rm GE}$ to $C_{\rm GE}$ and $Q$ is connected, either
$V(Q)=C_{\rm GE}$ or $V(Q)=D_{\rm GE}$. The first case gives a perfect
matching of $Q$. In the second case, $Q$ is factor-critical and hence has
odd order, contrary to the even order $m$.

Select the fixed edge $x_a y_a$ for every $a\in A_{\rm GE}$, and call the
resulting matching $K_A$. Choose a perfect matching of $Q[C_{\rm GE}]$ and
lift its edge instances in the standard way. Call this matching $K_C$.
At each index in $C_{\rm GE}$, one lift vertex is selected and the other is
dominated through the fixed edge.

Every neighbor of $W_i$ outside $W_i$ lies in $A_{\rm GE}$. Also,
$W_i$ has such a neighbor, since $Q$ is connected. Choose one edge between
$W_i$ and $A_{\rm GE}$, keep its direction in $\vec Q$, and apply
Lemma~\ref{lem:relative-switch}. We obtain a matching $K_i$ of
$(|W_i|-1)/2$ edges. Its endpoints, together with the selected vertices at
$A_{\rm GE}$, dominate all lift vertices indexed by $W_i$.

The index sets $A_{\rm GE}$, $C_{\rm GE}$, and
$W_1,\ldots,W_s$ are disjoint. Hence
$K=K_A\cup K_C\cup\bigcup_{i=1}^{s}K_i$ is a matching. Its endpoint set dominates the canonical lift. Its size is
\[
 \begin{aligned}
 |V(K)|
 &=2|A_{\rm GE}|+|C_{\rm GE}|+
   \sum_{i=1}^{s}(|W_i|-1)\\
 &=2|A_{\rm GE}|+|C_{\rm GE}|+|D_{\rm GE}|-s\\
 &=m-(s-|A_{\rm GE}|)=m-\delta.
 \end{aligned}
\]
The numbers $m$ and $\delta$ have the same parity. Since $Q$ has no
perfect matching, $\delta>0$. Thus $\delta\ge2$ and
$|V(K)|\le m-2$.
\end{proof}

\subsection{Deleting one arc}

\begin{lemma}
\label{lem:deleted-prescribed}
Let $\vec L$ be a legal balanced digraph of even order, let
$\epsilon=\alpha\to\beta$ be an arc, and put
$\vec Q=\vec L-\epsilon$. Let $Q$ be the underlying
multigraph of $\vec Q$. If $Q$ has a perfect matching $P$, then the
edge-deleted canonical lift $L(\vec L,\vec Q)$ has
\begin{enumerate}[label=\textnormal{(\roman*)}]
\item a dominating matching with $|P|$ edges that contains $x_\alpha$, and
\item a dominating matching with $|P|$ edges that contains $y_\beta$.
\end{enumerate}
In both matchings, the prescribed vertex is covered by a matching edge.
\end{lemma}

\begin{proof}
Write the edge instances of $P$ in their directions as
$P=\{t_i\to h_i:1\le i\le r\}$. If $\alpha$ is a matched tail, the standard lift of $P$ is already a
dominating matching and contains $x_\alpha$. We may therefore assume that
$\alpha=h_k$ for some matched arc $t_k\to h_k$.

Use the head-switching variables and the literal map from
Definition~\ref{def:switching-support}. Since $P$ is perfect,
$\lambda(u)$ is defined for every $u\in V(\vec Q)$. Consider
\[
 \Phi_Q=
 \xi_k\wedge
 \bigwedge_{i=1}^{r}
 \left(
  \neg\xi_i\vee
  \bigvee_{u\in\Nin_{\vec Q}(t_i)}\lambda(u)
 \right).
\]
This is the head-switching formula with the switch at $k$ prescribed. The
deleted arc gives no support in the affected tail clause.

Build the support graph with the single source arc $s\to k$, and let $R$
be its reachable positive part. Suppose that $\Phi_Q$ is not satisfiable.
Lemma~\ref{lem:support-obstruction} then shows that $R$ is acyclic and that
every negative target is reachable and strictly dominates the index of its
clause. By construction, $R$ is a finite rooted flow graph.

The deleted arc leaves $h_k=\alpha$. Thus $h_k$ has only one out-arc in
$\vec Q$. A support arc entering a variable $j$ records an arc of
$\vec Q$ leaving $h_j$. Legality implies that distinct ordinary support arcs entering $j$ have
distinct tails. The source tail $s$ is different from every variable, so
all arcs entering a vertex of $R$ have distinct tails. Including the source
arc, the indegree of $k$ in $R$ is at most two. Every other reachable
variable also has
indegree at most two. Let $\nu$ be the number of reachable variables, let
$E_+$ and
$E_-$ count positive and negative support occurrences in the reachable tail
clauses, and let $J$ be the number of reachable variables of indegree two.
Then $E_++1=\nu+J$.

Put $d=1$ if $\beta$ is a reachable matched tail, and put $d=0$ otherwise.
The deleted arc removes one support position from a reachable tail clause
exactly when $d=1$. All other support positions are literals whose targets
are reachable. Hence $E_++E_-=2\nu-d$, and it follows that
$E_-=\nu+1-J-d\ge\nu-J$.

A negative target $j$ occurs at most once, because the matched arc
$t_j\to h_j$ already uses one of the two arcs leaving $t_j$. The negative
targets are therefore distinct nonleaves of the dominator tree. The
hypotheses of Lemma~\ref{lem:dominator-leaves} hold for
$R$, so its dominator tree has at least $J+1$ variable leaves and at most
$\nu-J-1$ nonleaf variables. This contradicts
$E_-\ge\nu-J$. Thus $\Phi_Q$ is satisfiable.

Switch exactly the true variables. The resulting matching uses only fixed
edges and lifts of edges of $P$, so it does not use the deleted edge
$x_\alpha y_\beta$. Its clauses involve only arcs of $\vec Q$ and therefore
prove domination in $L(\vec L,\vec Q)$. Since $\xi_k=1$, the matching
contains $x_\alpha y_\alpha$ and hence covers $x_\alpha$.

For the second statement, reverse all arcs and interchange the two partite
classes. The first statement then gives a matching that covers
$y_\beta$ in the original lift.
\end{proof}

Combining the two preceding alternatives gives the form needed in the final
smallest-counterexample argument.

\begin{lemma}[Terminal alternative]\label{lem:terminal-alternative}
Let $H=(X,Y;E)$ be a finite simple connected cubic bipartite graph with
$|V(H)|\equiv0\pmod4$, and let $e=ab\in E(H)$, where $a\in X$ and
$b\in Y$. Put $m=|V(H)|/2$. Then one of the following holds.
\begin{enumerate}[label=\textnormal{(\alph*)}]
\item The graph $H-e$ has a paired dominating set of size at most $m-2$.
\item The graph $H-e$ has a paired dominating set of size $m$ in which
      $a$ is selected and matched inside $H-e$, and another such set in
      which $b$ is selected and matched inside $H-e$.
\end{enumerate}
In alternative \textnormal{(a)}, the paired dominating set induces a state
in $\mathcal G$ of the same cost. In alternative \textnormal{(b)},
$H-e$ has a state in $\mathcal G$ of cost $m$ whose first coordinate is
$\II$, and another whose second coordinate is $\II$.
\end{lemma}

\begin{proof}
Decompose $H$ into three perfect matchings, and choose one, say $M$, that
avoids $e$. Contract $M$ to obtain a legal balanced digraph $\vec L$ of
order $m$. Let $\epsilon=\alpha\to\beta$ be the arc corresponding to $e$,
and put $\vec Q=\vec L-\epsilon$. Let $Q$ be the underlying multigraph
of $\vec Q$.

The graph $H-e$ is connected by
Lemma~\ref{lem:cubic-bipartite-bridgeless}. Hence $Q$ is connected. If
$Q$ has no perfect matching, Theorem~\ref{thm:gallai-lift} gives a
dominating matching in $H-e$ whose endpoint set has size at most $m-2$.
By Proposition~\ref{prop:matching-form}, this endpoint set is a paired
dominating set, and alternative \textnormal{(a)} holds. If $Q$ has a
perfect matching, Lemma~\ref{lem:deleted-prescribed} gives two dominating
matchings, each with $m$ endpoints: one covers $x_\alpha=a$, and the
other covers $y_\beta=b$. Both matchings use only edges of $H-e$, and
Proposition~\ref{prop:matching-form} shows that their endpoint sets are
paired dominating sets of $H-e$.

A selected boundary vertex is covered by the matching and has symbol
$\II$. An unselected boundary vertex is dominated in $H-e$ and has symbol
$\DD$. Thus every state obtained above belongs to
$\mathcal G=\{\sII,\sID,\sDI,\sDD\}$, and the prescribed coordinate is
$\II$ in the second case.
\end{proof}

\begin{corollary}[Orders divisible by four]\label{cor:even-all}
Every finite simple connected cubic bipartite graph $G$ with
$|V(G)|\equiv0\pmod4$ satisfies $\gpr(G)\le |V(G)|/2$.
\end{corollary}

\begin{proof}
Choose a perfect matching of $G$ and contract it to a legal balanced
digraph $\vec L$ of even order $m=|V(G)|/2$. Put $\vec Q=\vec L$, and
let $Q$ be its connected underlying multigraph. Then
$L(\vec L,\vec Q)=L(\vec L)=G$. If $Q$ has a perfect matching, its
standard lift is a dominating matching with $m$ endpoints. Otherwise,
Theorem~\ref{thm:gallai-lift} gives a dominating matching with at most
$m-2$ endpoints. In either case, Proposition~\ref{prop:matching-form}
shows that the endpoint set is a paired dominating set of size at most
$m=|V(G)|/2$.
\end{proof}

\subsection{Proof of the main theorem}\label{sec:completion}

We finish the proof by combining the boundary states with the edge-deleted
alternative. We first summarize the outputs that will be used. If a basic
side has closure of order $c\equiv0\pmod4$, then it has both an $\sOO$
state and an $\sII$ state of cost at most $c/2$; if
$c\equiv2\pmod4$, it has a state in $\NonOpenStates$ of cost at most
$c/2-1$. For a simple terminal graph $H$ of order divisible by four,
Lemma~\ref{lem:terminal-alternative} gives either a state in $\mathcal G$
of cost at most $|V(H)|/2-2$, or two states in $\mathcal G$ of cost
$|V(H)|/2$ with the first and second coordinates prescribed to be $\II$.
The effect of every intervening two-vertex extension is given by
Lemma~\ref{lem:chain-bounds}. No further structural input is needed in the
case analysis below.

\begin{definition}\label{def:distinguished-reduction}
Let $J=(X,Y;E)$ be a connected cubic bipartite multigraph with a
marked edge $e^\star=xy$, where $x\in X$ and $y\in Y$. Suppose that
$e^\star$ is parallel to exactly one other edge. Let $y'$ be the third
neighbor of $x$, and let $x'$ be the third neighbor of $y$. Delete
$x,y$ and all incident edge instances, add the edge
$e_{\rm new}^\star=x'y'$, and make it the new marked edge. This
operation is called a \emph{marked-edge reduction}. Its inverse on the
side obtained after deleting the marked edge is a two-vertex extension.
\end{definition}

\begin{lemma}\label{lem:terminal-simple}
Let $G$ be a finite simple connected cubic bipartite graph, and let $C$ and
$D$ be the sides of a two-edge cut. Start with the closure
$D^+=D+e_D^\star$. Repeatedly applying the marked-edge reduction eventually gives a simple
connected cubic bipartite graph $H$. If the reduction is applied $t$ times and $e$
is the final marked edge, put $D_0=H-e$. Then $D$ is obtained from
$D_0$ by $t$ two-vertex extensions, and $|V(D)|=|V(H)|+2t$.
\end{lemma}

\begin{proof}
Both sides are connected. Otherwise, some component of a side would
have edge boundary at most one. Boundary zero contradicts the connectedness
of $G$, and boundary one contradicts
Lemma~\ref{lem:cubic-bipartite-bridgeless}. By
Lemma~\ref{lem:boundary-parts}, the marked edge of $D^+$ joins
opposite partite classes. Hence $D^+$ is a connected cubic bipartite
multigraph.

During the reductions, every edge other than the marked edge is inherited from the
simple graph $G$, and the marked edge is the only newly added edge.
Suppose one reduction is made at a parallel pair with ends $x\in X$ and
$y\in Y$. With the notation of
Definition~\ref{def:distinguished-reduction}, we have $x'\in X$ and
$y'\in Y$. Thus the new edge $x'y'$ is not a loop and joins opposite
partite classes. Removing $x,y$ lowers only the degrees of $x'$ and $y'$,
and the new edge restores both degrees to three.

The new graph is connected. Indeed, every component of
$J-\{x,y\}$ contains $x'$ or $y'$; otherwise it would already be a component
of $J$. The edge $x'y'$ joins the components containing $x'$ and $y'$.
Thus one reduction preserves connectedness, cubicity, and bipartiteness,
and also preserves the edge condition above.

Two inherited edges are never parallel. Hence the marked edge can
be parallel to at most one inherited edge. Whenever this happens, one
reduction removes two vertices. The process therefore stops. At the end,
the marked edge is not parallel to an inherited edge, so the final
graph $H$ is simple. It is also connected, cubic, and bipartite.

Deleting the final marked edge gives $D_0=H-e$. Reversing the
$t$ reductions gives $t$ two-vertex extensions from $D_0$ to $D$. Since each
reduction removes two vertices, the order formula follows.
\end{proof}

\begin{lemma}\label{lem:repair}
Suppose that local realizations on the two sides of a two-edge cut have
states $\sigma\in\NonOpenStates$ and $\tau\in\mathcal G$. If the
states are not compatible, then exactly one crossed coordinate joins $\RR$
to $\DD$. Selecting the two ends of that cut edge and adding the edge to
the global matching combines the two local realizations at an extra cost of
two.
\end{lemma}

\begin{proof}
Every symbol of $\tau$ is $\II$ or $\DD$. The only possible failure of
compatibility is therefore $\RR\not\sim\DD$. A state in
$\NonOpenStates$ contains at most one $\RR$, so there is exactly one bad
coordinate. Both ends of the corresponding cut edge are unselected.
Select them and match them through that edge. The old local matching edges
remain disjoint, every selected vertex is matched, and no domination
condition is lost.
\end{proof}

For an even integer $n$, put $B(n)=2\lfloor n/4\rfloor$.

\begin{proof}[Proof of Theorem~\ref{thm:main}]
Suppose that the theorem is false, and let $G$ be a counterexample of
minimum order $n$. By Corollary~\ref{cor:even-all}, we have
$n\equiv2\pmod4$ and $B(n)=n/2-1$.
The graph $G$ is not three-edge-connected by
Theorem~\ref{thm:intro-three-edge}. It has no bridge by
Lemma~\ref{lem:cubic-bipartite-bridgeless}. Hence $G$ has a two-edge cut.

Choose a side $C$ of minimum order among all sides of all two-edge cuts,
and let $D$ be the other side. By
Lemma~\ref{lem:minimum-side-basic}, the closure $F=C^+$ is a simple three-edge-connected cubic bipartite graph. Put
$c=|V(F)|$. Apply Lemma~\ref{lem:terminal-simple} to $D^+$. Let $H$ be
the final simple graph, let $e$ be its marked edge, let $t$ be the
number of reductions, and put $h=|V(H)|$. Then
$n=c+h+2t$.
Both $c$ and $h$ are even, since $F$ and $H$ are cubic bipartite graphs.
By Lemma~\ref{lem:terminal-simple}, $H$ is finite, simple, connected, cubic,
and bipartite. Since $h\le |V(D)|<n$, the minimality of $G$ applies to
$H$. In each case below, two compatible local realizations are joined by
Lemma~\ref{lem:two-cut-identity}, and their costs add.

The four possible congruence cases are
\[
\begin{array}{c|c|c}
 h\pmod4 & c\pmod4 & t\pmod2\\
\hline
 2&0&0\\
 2&2&1\\
 0&0&1\\
 0&2&0
\end{array}
\]
They follow from $n=c+h+2t\equiv2\pmod4$.

We first prove that $h\equiv0\pmod4$. Suppose instead that $h\equiv2\pmod4$. By the choice of $G$,
$\gpr(H)\le h/2-1$. The closure formula~\eqref{eq:closure-formula} gives a
state $\tau\in\ClStates$ in $H-e$ of cost at most $h/2-1$.

\medskip
\noindent\textbf{Case 1.} $c\equiv0\pmod4$.
Then $n=c+h+2t$ shows that $t$ is even. By
Lemma~\ref{lem:chain-bounds}\textnormal{(i)}, the side $D$ has a state in
$\mathcal U$ of cost at most $h/2-1+t=|V(D)|/2-1$. Use the $\sOO$ realization on $C$ when this state is $\sOO$, and use the
$\sII$ realization otherwise. Lemma~\ref{lem:even-basic-profile} gives
both realizations at cost at most $c/2$. The states are compatible, and
the total cost is at most $n/2-1$.

\medskip
\noindent\textbf{Case 2.} $c\equiv2\pmod4$.
Then $t$ is odd. By Lemma~\ref{lem:odd-basic-profile}, $C$ has a state
$\sigma\in\NonOpenStates$ of cost at most $c/2-1$. By
Lemma~\ref{lem:chain-bounds}\textnormal{(ii)}, $D$ has a compatible state of
cost at most $(h/2-1)+t+1=|V(D)|/2$. The total cost is again at most $n/2-1$.

Both cases contradict the choice of $G$. Thus $h\equiv0\pmod4$.

Apply Lemma~\ref{lem:terminal-alternative} to $H$ and $e$. In
alternative \textnormal{(a)}, keep the state of cost at most $h/2-2$.
In alternative \textnormal{(b)}, keep both states of cost $h/2$, one with
first coordinate $\II$ and one with second coordinate $\II$. In
particular, $H-e$ has a state in $\mathcal G$ of cost at most $h/2$.

\medskip
\noindent\textbf{Case 3.} $c\equiv0\pmod4$.
The identity $n=c+h+2t$ shows that $t$ is odd. By
Lemma~\ref{lem:chain-bounds}\textnormal{(v)}, the terminal state extends to
a state in $\mathcal S$ on $D$ at cost at most $h/2+t-1$. The $\sII$ state on $C$ is compatible with every state in $\mathcal S$.
By Lemma~\ref{lem:even-basic-profile}, it has cost at most $c/2$. The
total cost is therefore at most $c/2+h/2+t-1=n/2-1$, a contradiction.

\medskip
\noindent\textbf{Case 4.} $c\equiv2\pmod4$.
Now $t$ is even. Fix a state $\sigma\in\NonOpenStates$ on $C$ with cost at
most $c/2-1$, as given by Lemma~\ref{lem:odd-basic-profile}.

Suppose first that $t\ge2$. Choose any terminal state in $\mathcal G$ of
cost at most $h/2$. By Lemma~\ref{lem:chain-bounds}\textnormal{(v)}, it
extends to the state $\sII$ on $D$ at cost at most $h/2+t$. This state is
compatible with every state in $\NonOpenStates$. Hence the total cost is at
most $(c/2-1)+(h/2+t)=n/2-1$, a contradiction.

It remains to consider $t=0$. If
Lemma~\ref{lem:terminal-alternative}\textnormal{(a)} holds, choose a state
$\tau\in\mathcal G$ of cost at most $h/2-2$. If $\sigma$ and $\tau$ are
compatible, join the two realizations directly. Otherwise use
Lemma~\ref{lem:repair}. Even with the extra two vertices, the total cost is
at most $(c/2-1)+(h/2-2)+2=n/2-1$, a contradiction.

Finally suppose that
Lemma~\ref{lem:terminal-alternative}\textnormal{(b)} holds. If $\sigma$
contains no $\RR$, it is compatible with every state in $\mathcal G$; take
either terminal realization. If $\sigma=\sIR$, use the terminal
realization whose first coordinate is $\II$. If $\sigma=\sRI$, use the
one whose second coordinate is $\II$. The crossed-coordinate convention
from Definition~\ref{def:crossed-coordinates} shows that the chosen states
are compatible. Their total cost is at most $(c/2-1)+h/2=n/2-1$, again a
contradiction.

In all four cases, the two sides have
compatible local realizations whose total cost is at most $B(n)$. Therefore
no counterexample exists, and the theorem follows.
\end{proof}

%\section{Conclusion}\label{sec:conclusion}

%We have proved the Desormeaux--Henning conjecture in the sharp integer
%form $\gpr(G)\le2\lfloor |V(G)|/4\rfloor$ for every finite simple cubic
%bipartite graph $G$. The proof starts with a fixed perfect matching and the
%legal balanced digraph obtained by contracting its edges. A dominator-tree
%count supports the switching arguments in the three-edge-connected case.
%For a two-edge cut, four boundary symbols and an exact transfer table join
%the two sides. Finally, a relative switching lemma and the Gallai--Edmonds
%decomposition provide the two-vertex saving needed in the smallest-counterexample
%argument. The examples $K_{3,3}$ and $Q_3$ attain the bound in the two
%possible congruence classes modulo four.

\section*{Declaration on the use of AI}
The authors used ChatGPT 5.6 Pro to assist in discussing  proof strategies, checking proofs, and improving exposition.

\appendix
\section{Verification of the state-transfer table}\label{app:bond-check}

We use the notation from Theorem~\ref{thm:exact-state-transfer}. The table
below lists every legal choice. The second column gives the selected
subset of the new boundary vertices $\{u,v\}$. The third column gives the
new matching edges $L\subseteq\{uv,uy,vx\}$. The fourth column is the
resulting outer state. The last column is the number of newly selected
vertices.

The list is complete for the following reasons. An inner $\OO$ at $x$
forces $vx\in L$, while an inner $\II$, $\DD$, or $\RR$ at $x$ forbids
$vx$. The corresponding statement at $y$ holds for $uy$. An inner
request at $x$ forces $v$ to be selected, and an inner request at $y$
forces $u$ to be selected. Thus conditions \textnormal{(T1)}--\textnormal{(T3)}
first determine the allowed incidences of $vx$ and $uy$ and the compulsory
members of the selected set. Among the remaining possibilities,
\textnormal{(T1)} determines whether $uv$ may belong to $L$, and
\textnormal{(T4)} uniquely determines both outer symbols. Hence the table
exhausts all $4\cdot8$ choices for each inner state and gives the $45$
entries listed in Table~\ref{tab:full-state-transfer}.

\begingroup
\small
\setlength{\tabcolsep}{5pt}
\renewcommand{\arraystretch}{1.08}
\begin{longtable}{c c c c c}
\caption{All legal choices in one two-vertex extension.}\label{tab:bond-witnesses}\\
\toprule
Inner state & Selected new vertices & New matching $L$ & Outer state & $d$\\
\midrule
\endfirsthead
\toprule
Inner state & Selected new vertices & New matching $L$ & Outer state & $d$\\
\midrule
\endhead
\midrule
\multicolumn{5}{r}{Continued on the next page}\\
\endfoot
\bottomrule
\endlastfoot
$\sOO$ & $\{u,v\}$ & $\{uy,vx\}$ & $\sII$ & $2$\\
\addlinespace[2pt]
$\sOI$ & $\{v\}$ & $\{vx\}$ & $\sDI$ & $1$\\
 & $\{u,v\}$ & $\{vx\}$ & $\sOI$ & $2$\\
\addlinespace[2pt]
$\sOD$ & $\{v\}$ & $\{vx\}$ & $\sDI$ & $1$\\
 & $\{u,v\}$ & $\{vx\}$ & $\sOI$ & $2$\\
\addlinespace[2pt]
$\sOR$ & $\{u,v\}$ & $\{vx\}$ & $\sOI$ & $2$\\
\addlinespace[2pt]
$\sIO$ & $\{u\}$ & $\{uy\}$ & $\sID$ & $1$\\
 & $\{u,v\}$ & $\{uy\}$ & $\sIO$ & $2$\\
\addlinespace[2pt]
$\sII$ & $\varnothing$ & $\varnothing$ & $\sDD$ & $0$\\
 & $\{v\}$ & $\varnothing$ & $\sDO$ & $1$\\
 & $\{u\}$ & $\varnothing$ & $\sOD$ & $1$\\
 & $\{u,v\}$ & $\varnothing$ & $\sOO$ & $2$\\
 & $\{u,v\}$ & $\{uv\}$ & $\sII$ & $2$\\
\addlinespace[2pt]
$\sID$ & $\varnothing$ & $\varnothing$ & $\sRD$ & $0$\\
 & $\{v\}$ & $\varnothing$ & $\sDO$ & $1$\\
 & $\{u\}$ & $\varnothing$ & $\sOD$ & $1$\\
 & $\{u,v\}$ & $\varnothing$ & $\sOO$ & $2$\\
 & $\{u,v\}$ & $\{uv\}$ & $\sII$ & $2$\\
\addlinespace[2pt]
$\sIR$ & $\{u\}$ & $\varnothing$ & $\sOD$ & $1$\\
 & $\{u,v\}$ & $\varnothing$ & $\sOO$ & $2$\\
 & $\{u,v\}$ & $\{uv\}$ & $\sII$ & $2$\\
\addlinespace[2pt]
$\sDO$ & $\{u\}$ & $\{uy\}$ & $\sID$ & $1$\\
 & $\{u,v\}$ & $\{uy\}$ & $\sIO$ & $2$\\
\addlinespace[2pt]
$\sDI$ & $\varnothing$ & $\varnothing$ & $\sDR$ & $0$\\
 & $\{v\}$ & $\varnothing$ & $\sDO$ & $1$\\
 & $\{u\}$ & $\varnothing$ & $\sOD$ & $1$\\
 & $\{u,v\}$ & $\varnothing$ & $\sOO$ & $2$\\
 & $\{u,v\}$ & $\{uv\}$ & $\sII$ & $2$\\
\addlinespace[2pt]
$\sDD$ & $\varnothing$ & $\varnothing$ & $\sRR$ & $0$\\
 & $\{v\}$ & $\varnothing$ & $\sDO$ & $1$\\
 & $\{u\}$ & $\varnothing$ & $\sOD$ & $1$\\
 & $\{u,v\}$ & $\varnothing$ & $\sOO$ & $2$\\
 & $\{u,v\}$ & $\{uv\}$ & $\sII$ & $2$\\
\addlinespace[2pt]
$\sDR$ & $\{u\}$ & $\varnothing$ & $\sOD$ & $1$\\
 & $\{u,v\}$ & $\varnothing$ & $\sOO$ & $2$\\
 & $\{u,v\}$ & $\{uv\}$ & $\sII$ & $2$\\
\addlinespace[2pt]
$\sRO$ & $\{u,v\}$ & $\{uy\}$ & $\sIO$ & $2$\\
\addlinespace[2pt]
$\sRI$ & $\{v\}$ & $\varnothing$ & $\sDO$ & $1$\\
 & $\{u,v\}$ & $\varnothing$ & $\sOO$ & $2$\\
 & $\{u,v\}$ & $\{uv\}$ & $\sII$ & $2$\\
\addlinespace[2pt]
$\sRD$ & $\{v\}$ & $\varnothing$ & $\sDO$ & $1$\\
 & $\{u,v\}$ & $\varnothing$ & $\sOO$ & $2$\\
 & $\{u,v\}$ & $\{uv\}$ & $\sII$ & $2$\\
\addlinespace[2pt]
$\sRR$ & $\{u,v\}$ & $\varnothing$ & $\sOO$ & $2$\\
 & $\{u,v\}$ & $\{uv\}$ & $\sII$ & $2$\\
\end{longtable}
\endgroup

\end{document}